\documentclass[11pt]{amsart}

\usepackage{amssymb}
\usepackage{amsmath}
\usepackage{amsfonts}
\usepackage{graphicx}
\usepackage{amsthm}
\usepackage{enumerate}
\usepackage[mathscr]{eucal}
\usepackage{mathrsfs}
\usepackage{verbatim}
\usepackage{yhmath}
\usepackage{epstopdf}
\usepackage{color}

\makeatletter
\@namedef{subjclassname@2010}{%
  \textup{2010} Mathematics Subject Classification}
\makeatother

\numberwithin{equation}{section}
\numberwithin{figure}{section}

\theoremstyle{plain}
\newtheorem{theorem}{Theorem}[section]
\newtheorem{lemma}[theorem]{Lemma}
\newtheorem{proposition}[theorem]{Proposition}
\newtheorem{corollary}[theorem]{Corollary}

\theoremstyle{plain}

\theoremstyle{remark}
\newtheorem{remark}[theorem]{Remark}

\begin{document}
\date{}

\title
[Eigenvalues and lattice points]{The Weyl formula for planar annuli}

\author{Jingwei Guo}
\address{School of Mathematical Sciences\\
University of Science and Technology of China\\
Hefei, 230026\\ P.R. China\\}
\email{jwguo@ustc.edu.cn}

\author{Wolfgang M\"uller}
\address{Institut of Statistics\\ Graz University of Technology\\ 8010 Graz\\ Austria}
\email{w.mueller@tugraz.at}

\author{Weiwei Wang}
\address{Yau Mathematical Sciences Center\\Tsinghua University\\Beijing\\100084, P. R. China\\
} \email{wawnwg123@163.com}

\author{Zuoqin Wang}
\address{School of Mathematical Sciences\\
University of Science and Technology of China\\
Hefei, 230026\\ P.R. China\\} \email{wangzuoq@ustc.edu.cn}

\thanks{J.G. is partially supported by the NSFC Grant (No. 11501535 and 11571331) and the Fundamental Research Funds for the
Central Universities No. WK3470000013. Z.W. is partially supported by the NSFC Grant (No. 11721101 and 11571331).
}



\begin{abstract}
We study the zeros of cross-product of Bessel functions and obtain their approximations, based on which we reduce the eigenvalue counting problem for the Dirichlet Laplacian associated with a planar annulus to a lattice point counting problem associated with a special domain in $\mathbb{R}^2$. Unlike other lattice point problems, the one arisen naturally here has interesting features that lattice points under consideration are translated by various amount and the curvature of the boundary is unbounded. By transforming this problem into a relatively standard form and using classical van der Corput's bounds, we obtain a two-term Weyl formula for the eigenvalue counting function for the planar annulus with a remainder of size $O(\mu^{2/3})$. If we additionally assume that certain tangent has rational slope, we obtain an improved remainder estimate of the same strength as Huxley's bound in the Gauss circle problem, namely $O(\mu^{131/208}(\log \mu)^{18627/8320})$. As a by-product of our lattice point counting results, we readily obtain this Huxley-type remainder estimate in the two-term Weyl formula for planar disks.
\end{abstract}

\maketitle

\section{Introduction}

Let $D \subset \mathbb R^2$ be a bounded domain with piecewise smooth boundary, and let
\[0< \mu^2_1 < \mu^2_2 \le \mu_3^2 \le \cdots \]
be the eigenvalues (counting multiplicity) of the Dirichlet Laplacian associated with $D$. In his seminal work \cite{weyl11:1912}, H. Weyl initiated the study of the asymptotic behavior of the eigenvalue counting function
\begin{equation}
\mathscr{N}_D(\mu)=\#\left\{k\in\mathbb{N} : \mu_k\leq \mu \right\},\label{e-counting}
\end{equation}
and he proved that as $\mu \to \infty$,
\begin{equation*}
\mathscr{N}_D(\mu)=\frac{\mathrm{Area}(D)}{4\pi}\mu^2+o(\mu^2).
\end{equation*}
If we interpret $\mu_j$'s as the frequencies, i.e. the overtones that can be produced by a drum whose drumhead has the shape $D$, then the Weyl's law mentioned above implies that one can ``hear'' the area of $D$.
Weyl further conjectured that one can also ``hear'' the perimeter of $D$. More precisely, he conjectured in 1913 (see \cite{weyl13:1913}) that
\begin{equation}
\mathscr{N}_D(\mu)=\frac{\mathrm{Area}(D)}{4\pi}\mu^2-\frac{\mathrm{Length}(\partial D)}{4\pi}\mu+o(\mu). \label{WeylConj}
\end{equation}
Since then, the asymptotic behavior of the eigenvalue counting function has been studied extensively by many mathematicians in many different settings.
For example,  for closed manifolds  the Weyl's conjecture was first proven by J. Duistermaat and V. Guillemin \cite{DuG:1975} under an extra assumption that the set of periodic bicharacteristics has measure zero (which turns out to be necessary in this setting). Their result was later generalized by V. Ivrii \cite{Ivrii1980} (see also R. Melrose \cite{Mel:1980}) to manifolds with boundary under a similar assumption that the set of periodic billiard trajectories has measure zero.

While it is still unknown whether the conjecture is true for all bounded domains in $\mathbb R^2$ with piecewise smooth boundaries, it is known that many regions, including all bounded convex domains with analytic boundaries and all bounded domains with piecewise smooth concave boundaries, do satisfy the condition on the periodic billiard trajectories and thus the two-term Weyl's law \eqref{WeylConj}. In terms of such a region $D$, natural questions are: what other information is encoded in $\mathscr{N}_D(\mu)$ and is there a third main term in the asymptotics of  $\mathscr{N}_D(\mu)$? It turns out that the answers are no in general. In fact, V. Lazutkin and D. Terman \cite{Lazu:1982} showed that for any $\kappa <1$, there exists a convex planar domain $D$ satisfying \eqref{WeylConj} with an error term  of order at least $\mu^\kappa$. In other words, if one sets
\begin{equation}
\mathscr{R}_D(\mu)=\mathscr{N}_D(\mu) - \frac{\mathrm{Area}(D)}{4\pi}\mu^2+\frac{\mathrm{Length}(\partial D)}{4\pi}\mu,
\end{equation}
then, for any  $\kappa <1$, $\mathscr{R}_D(\mu) \ne O(\mu^\kappa)$ for some convex planar domain $D$.   Due to the complexity of the dynamics of the billiard flow, there is in general no hope to get a universal estimate of  $\mathscr{R}_D(\mu)$ better than $o(\mu)$.

On the other hand, there are many domains (usually with special symmetry) for which one can prove a much better estimate of $\mathscr{R}_D(\mu)$. Such examples include squares, disks, ellipses and, in principle, regions of separable variable type. Since the Laplacian eigenvalues  for them are closely related to lattice points,  a basic strategy is to convert the eigenvalue counting problem to a lattice point counting problem associated with some special planar region (modulo an error which needs to be controlled). For example, it is easy to see that the Laplacian eigenvalues of the planar unit square are in one-to-one correspondence with integer points in the first quadrant, and thus the eigenvalue counting problem is equivalent to the famous Gauss circle problem, which has received much attention for more than one hundred years while the conjectured error-term estimate $O(\mu^{1/2+\varepsilon})$ is still far from being proved. In principle the same type of arguments can be applied to  other domains whose billiard flows are completely integrable.

Recently the same idea was applied by Y. Colin de Verdi\`ere \cite{colin:2011} to get a nicer estimate of $\mathscr{R}_D(\mu)$ for planar disks. By studying the asymptotics of the Bessel function, Colin de Verdi\`ere converted the eigenvalue counting problem to a lattice point counting problem associated with a special planar domain with two cusp points. By using tools from analysis, he showed that both the error term in the lattice point problem and the error between the eigenvalue and lattice counting functions  are of order $O(\mu^{2/3})$. As a consequence, he proved $\mathscr{R}_{disk}(\mu)=O(\mu^{2/3})$. The same result was also obtained by N. Kuznetsov and  B. Fedosov \cite{kuz:1965}.

In \cite{GWW2018} three of the authors followed Colin de Verdi\`ere's strategy to study disks and observed that the error  between the eigenvalue and lattice counting functions is controlled by the error term in the corresponding lattice point problem. By applying the van der Corput's method of estimating exponential sums in the latter problem, we were able to improve Colin de Verdi\`ere's result a little bit and prove that  $\mathscr{R}_{disk}(\mu)=O(\mu^{2/3-1/495})$.

However, comparing to known results for squares, the exponent $2/3-1/495$ that we obtained for disks seems to be far from optimal. As we mentioned above, the eigenvalue counting problem for the unit square is equivalent to the Gauss circle problem, which is about counting lattice points in planar disks. So far the best published bound $O(\mu^{131/208}(\log \mu)^{18627/8320})$ is given by M. Huxley in \cite{Huxley:2003}.\footnote{In a recent preprint \cite{BWpreprint}, J. Bourgain and N. Watt was able to improve Huxley's bound in the circle problem to $O(\mu^{517/824+\varepsilon})$ by combining a newly emerging theory from harmonic analysis.}

This paper can be viewed as a continuation of \cite{colin:2011} and \cite{GWW2018} in two aspects: improving previous results for disks to a Huxley-type remainder estimate and extending from disks to annuli.

In the rest of this paper we let
\begin{equation*}
\mathscr{D}=\{x\in\mathbb{R}^2 : r\leq |x|\leq R \}
\end{equation*}
be the annulus centered at the origin with two given radii $0<r<R<\infty$. We obtain the following estimates of $\mathscr{R}_{\mathscr{D}}(\mu)$:
\begin{theorem}\label{specthm}
The following two-term Weyl formula for the annulus $\mathscr{D}$
\begin{equation*}
\mathscr{N}_\mathscr{D}(\mu)=\frac{R^2-r^2}{4}\mu^2-\frac{R+r}{2}\mu+O\left(\mu^{2/3}\right)
\end{equation*}
holds. Furthermore, if $\pi^{-1}\arccos(r/R)\in \mathbb{Q}$ then the remainder estimate can be improved to
\begin{equation*}
\mathscr{R}_{\mathscr{D}}(\mu)=O\left(\mu^{131/208}(\log \mu)^{18627/8320} \right).
\end{equation*}
\end{theorem}

\begin{remark}
The number $\pi^{-1}\arccos(r/R)$ represents the slope of certain tangent line in the associated lattice point problem.
If it is rational then we can apply Huxley's bounds for rounding error sums from \cite{Huxley:2003} in the estimation of the number of lattice points. For details see Section \ref{LatticeSec}. At this moment we are still not sure  whether this assumption can be removed or not if one wants the same Huxley-type bound.
\end{remark}

\begin{remark}
Roughly speaking, as in \cite{colin:2011}, the proof of this theorem consists of two parts:  first,  a reduction from the eigenvalue counting problem to certain lattice counting problem; second, study of the latter problem.  However, both parts in the annulus case are more complicated than their counterparts in the disk case.

In the first part, we need to find approximations of zeros of cross-product of Bessel functions of the first and second type with errors under good control (see Corollary \ref{approximation}) while in the disk case we only need to study zeros of the Bessel function of the first kind. In the derivation we discuss in four cases depending on the sizes of zeros and use expansions of Bessel functions given by the method of stationary phase and F. Olver \cite{olver:1954}. As a by-product, we obtain estimates of distances between adjacent zeros (see Corollary \ref{cor1}). Based on the approximations, we get a correspondence between eigenvalues and lattice points (translated by various amount), via which the eigenvalue counting problem is reduced to a lattice counting problem naturally. For details see Section \ref{zeros} and \ref{reduction-sec}.

In the second part, in order to solve this new lattice point problem, we translate (if necessary) the lattice points to achieve a uniformity in translation. Then we are led to study two lattice point problems with unbounded curvature and (possibly) cusps. Some boundary points with infinite curvature have tangents with rational slope (for example, the points $P_2$ and  $P_2'$ in Figure \ref{SymmH}).  These points are relatively easier to handle. This phenomenon occurs in the disk case. What is different in the annulus case is that for some boundary points with infinite curvature we do not know whether the slopes of their tangents are rational or not (for example, the points $J$ and $J'$ in Figure \ref{SymmH}). These points bring us troubles in the estimation. With rational slopes we apply Huxley's \cite[Proposition 3]{Huxley:2003}. Concerning the irrational case, Huxley's proposition does not seem to be applicable. For details see Section \ref{reduction-sec} and \ref{LatticeSec}.
\end{remark}


As to disks, heuristically, since $\pi^{-1}\arccos(0)\in \mathbb{Q}$ letting $r \to 0$ in Theorem \ref{specthm} leads to the following Huxley-type remainder estimate of $\mathscr{R}_{disk}(\mu)$, which improves the main results in \cite{GWW2018}. Its rigorous proof relies on the reduction step from the eigenvalue counting to the lattice point counting (see \cite[Section 3]{colin:2011}, \cite[Section 6]{GWW2018} and its variant in Section \ref{reduction-sec}),  Theorem \ref{theorem:no-in-D} (together with the symmetry of the domain $\mathcal{D}$) and the fact that the corresponding domain for the lattice point counting (Figure 1.1 in \cite{colin:2011}) is invariant under the involution $(x,y)\rightarrow (-x, y+x)$ (see \cite[P.3]{colin:2011}).

\begin{theorem}\label{diskcase}
For planar disks we have
\begin{equation*}
\mathscr{R}_{disk}(\mu)=O\left(\mu^{131/208}(\log \mu)^{18627/8320} \right).
\end{equation*}
\end{theorem}


\emph{Notations:} As in 3.6.15 \cite[P.15]{abram:1972} we write
\begin{equation*}
f(x)\sim \sum_{k=0}^{\infty} a_k x^{-k},
\end{equation*}
if
\begin{equation*}
f(x)-\sum_{k=0}^{n-1} a_k x^{-k}=O(x^{-n}) \quad \textrm{as $x\rightarrow \infty$}
\end{equation*}
for every $n=1, 2, \ldots.$

For functions $f$ and $g$ with $g$ taking nonnegative real values,
$f\ll g$ means $|f|\leqslant Cg$ for some constant $C$. If $f$
is nonnegative, $f\gg g$ means $g\ll f$. The notation
$f\asymp g$ means that $f\ll g$ and $g\ll f$. If we write a subscript (for instance $\ll_{\sigma}$), we emphasize that the implicit constant depends on that specific subscript.


\section{Zeros of cross-product of Bessel functions}\label{zeros}

There are a lot of literature on the study of the zeros of cross-product of Bessel functions. Just to mention a few, M. Kline~\cite{Kline:1948}, D. Willis~\cite{willis:1965}, J. Cochran~\cite{cochran:1964, Cochran:1966}, V. Bobkov~\cite{bobkov:2018}, etc. In this section we study such objects from our own perspectives (motivated by the work in \cite{colin:2011}), via whose study we look for a two-term Weyl formula for planar annuli.

Let $0<r<R<\infty$ be two given numbers. For any nonnegative integer $n$ we would like to study zeros of the function
\begin{equation}
f_n(x):=J_n(Rx)Y_n(rx)-J_n(rx)Y_n(Rx),\label{eigenequation}
\end{equation}
where $J_n$ and $Y_n$ are the Bessel functions of the first and second kind and order $n$ (see \cite[p. 360]{abram:1972}). It is well-known that all zeros are real and simple, and that $f_n$ is an even function. Hence we only study positive zeros. For each nonnegative $n$ we denote its sequence of positive zeros by $0<x_{n, 1}<x_{n, 2}<\cdots<x_{n, k}<\cdots$. In fact $x_{n,k}$ is strictly increasing in $n$ for each fixed $k\in \mathbb{N}$ (see \cite[P.425]{willis:1965}).

Throughout this paper we denote by $g$ the function
\begin{equation}
g(x)=\left(\sqrt{1-x^2}-x\arccos x\right)/\pi   \label{def-g}
\end{equation}
and by $G$ the function
\begin{equation*}
G(x)=\left\{
\begin{aligned}
&Rg(x/R)-rg(x/r)\;\; &\mathrm{for}&\;0\leq x\leq r,\\
&Rg(x/R)\;\; &\mathrm{for}&\;r\leq x\leq R.
\end{aligned}
\right.
\end{equation*}
In Figure \ref{H}, the solid curve $P_1JP_2$ represents the graph of $G$  on $[0, R]$, while the half-dashed and half-solid curve $P_{0}JP_{2}$  represents the graph of $Rg(x/R)$ on $[0, R]$.

\begin{figure} [ht]
\centering
\includegraphics[width=0.6\textwidth]{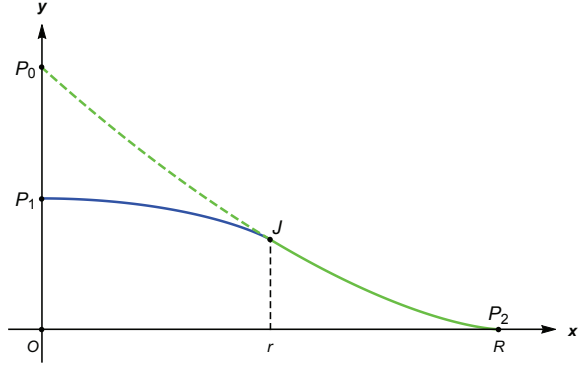}
\caption{The graph of $G$ with $R=2$ and $r=1$.}
\label{H}
\end{figure}


\begin{lemma}\label{case111}
For any $c>0$ and all $n\in \mathbb{N}\cup \{0\}$, if $rx\geq \max\{(1+c)n, 1\}$ then
\begin{equation}
f_n(x)=-\frac{2}{\pi}\frac{\sin\left( \pi x G\left(\frac{n}{x}\right)\right)+E_1(x)}{\left(\left(Rx\right)^2-n^2\right)^{1/4}
\left(\left(rx\right)^2-n^2\right)^{1/4}},          \label{case111-1}
\end{equation}
where
\begin{equation*}
E_1(x)=O_c\left(x^{-1}\right).
\end{equation*}
\end{lemma}

\begin{proof}
This result follows from an application of the method of stationary phase to the Bessel functions. More precisely, we first apply to all four factors in \eqref{eigenequation} the asymptotics \eqref{jnasy} and \eqref{ynasy} of Bessel functions and then use the angle difference formula for the sine.
\end{proof}


\begin{lemma}\label{case222}
There exists a constant $c\in (0,1)$ such that for any $\varepsilon>0$ and all sufficiently large $n$ if $n+n^{1/3+\varepsilon}\leq rx\leq (1+c)n$ then
\begin{equation}
f_n(x)=-\frac{2}{\pi}\frac{\sin\left( \pi x G\left(\frac{n}{x}\right)\right)+E_2(x)}{\left(\left(Rx\right)^2-n^2\right)^{1/4}
\left(\left(rx\right)^2-n^2\right)^{1/4}},           \label{case222-1}
\end{equation}
where
\begin{equation}
E_2(x)=O\left(z^{-3/2}\right)          \label{case222-4}
\end{equation}
with $z$ determined by the equation $rx=n+z n^{1/3}$.
\end{lemma}

\begin{proof}
Denote
\begin{equation*}
Rx=n z_{R} \quad \textrm{and}\quad  rx=n z_{r}.
\end{equation*}
For sufficiently large $n$ we apply to all four factors in \eqref{eigenequation} Olver's asymptotic expansions of Bessel functions of large order (see \cite[p. 368]{abram:1972} or Olver's original paper \cite{olver:1954}; for the convenience of the readers we put those formulas in the appendix). The $\zeta_R=\zeta(z_R)$ and $\zeta_r=\zeta(z_r)$ appearing in the asymptotics are both negative and determined by \eqref{def-zeta1}. They satisfy the following size estimates:
\begin{equation*}
(-\zeta_R)^{3/2}\asymp 1
\end{equation*}
and
\begin{equation*}
n^{-1+3\varepsilon/2}\ll (-\zeta_r)^{3/2}\ll 1
\end{equation*}
whenever $n$ is sufficiently large. Indeed, the first estimate follows from $R/r<z_R\leq R(1+c)/r$ (with $c\in (0,1)$ to be determined below) while the second one follows from \eqref{bound-zeta+} and $1+n^{-2/3+\varepsilon}\leq z_r\leq 1+c$.

Then
\begin{equation}
\begin{split}
f_n(x)=&-2 n^{-2/3}(-\zeta_R)^{1/4}(-\zeta_r)^{1/4}\left(z_R^2-1\right)^{-1/4} \left(z_r^2-1\right)^{-1/4}\cdot \\
       &\left[\mathrm{Ai}(n^{2/3}\zeta_R)\mathrm{Bi}(n^{2/3}\zeta_r)-\mathrm{Ai}(n^{2/3}\zeta_r)\mathrm{Bi}(n^{2/3}\zeta_R)
+E\right]
\end{split} \label{case222-2}
\end{equation}
with the error $E$ being an expression involving the Airy functions of the first and second kind.

Since
\begin{equation*}
n^{2/3}(-\zeta_R)\asymp n^{2/3},  \quad\quad n^{\varepsilon}\ll n^{2/3}(-\zeta_r)\ll n^{2/3},
\end{equation*}
$\mathrm{Ai}(-r)$ and $\mathrm{Bi}(-r)$ are both of size $O(r^{-1/4})$ while $\mathrm{Ai}'(-r)$ and $\mathrm{Bi}'(-r)$ of size $O(r^{1/4})$ (see \cite[p. 448--449]{abram:1972}), by using the well-known asymptotics for the Airy functions (see the appendix) and the angle difference formula for the sine, the terms in brackets in \eqref{case222-2} become
\begin{equation*}
\bigg[\frac{(-\zeta_R)^{-1/4}(-\zeta_r)^{-1/4}}{\pi n^{1/3}}\left(\sin\left(\frac{2}{3}n(-\zeta_R)^{3/2}-
\frac{2}{3}n(-\zeta_r)^{3/2}\right)+E_2(x)\right)\bigg],
\end{equation*}
where
\begin{equation*}
E_2(x)=O\left(|n^{2/3}\zeta_r|^{-3/2}+n^{-1}\right).
\end{equation*}
By  \eqref{bound-zeta+}, if $c$ is a sufficiently small constant then $|n^{2/3}\zeta_r|\asymp z$. Noticing the definition of $G$ and $z\leq cn^{2/3}$, we then get \eqref{case222-1} and \eqref{case222-4}.
\end{proof}


\begin{lemma} \label{case2.5}
There exists a strictly decreasing real-valued $C^1$ function $\psi: \mathbb{R} \rightarrow (0, 1/4)$ such that $\psi(0)=1/12$,  $\lim_{x\rightarrow -\infty}\psi=1/4$, $\lim_{x\rightarrow \infty}\psi=0$, and the image of $\psi'$ is a bounded interval. For any $\varepsilon>0$ and all sufficiently large $n$ if $n-n^{1/3+\varepsilon}\leq rx \leq n+n^{1/3+\varepsilon}$ then
\begin{equation}
f_n(x)=-\frac{2^{5/6}}{\pi^{1/2}}\frac{\sin\left(\pi xG\left(\frac{n}{x}\right)+\pi \psi\left(z\right)\right)+E_3(x)}{n^{1/3}\left(\left(Rx\right)^2-n^2\right)^{1/4}
\left(\mathrm{Ai}^2+\mathrm{Bi}^2\right)^{-1/2}\left(-2^{1/3}z\right)},\label{case2.5-1}
\end{equation}
where $z$ is determined by the equation $rx=n+z n^{1/3}$ and
\begin{equation}
E_3(x)=O\left(n^{-2/3+2.5\varepsilon}\right). \label{case2.5-5}
\end{equation}
\end{lemma}

\begin{proof}
Notice that if $rx\geq n-n^{1/3+\varepsilon}$ then
\begin{equation*}
Rx\geq \frac{R}{r}\left(n-n^{1/3+\varepsilon}\right)>\left(1+c'\right)n
\end{equation*}
for some fixed constant $c'>0$ whenever $n$ is sufficiently large. Denote
\begin{equation*}
rx=n+zn^{1/3} \quad  \textrm{with $-n^\varepsilon\leq z\leq n^\varepsilon$}.
\end{equation*}
Applying \eqref{jnasy} and \eqref{ynasy} to $J_n(Rx)$ and $Y_n(Rx)$ respectively and
Lemma \ref{9.3.4analogue} to both $J_n(rx)$ and $Y_n(rx)$ yields
\begin{align}
f_n(x)=&-2^{5/6}\pi^{-1/2}\left(\left(Rx\right)^2-n^2\right)^{-1/4}n^{-1/3}\sqrt{\mathrm{Ai}^2+\mathrm{Bi}^2}(-2^{1/3}z) \cdot \nonumber \\
       &\bigg[\sin\left( \pi x R g\left(\frac{n/x}{R}\right)-\frac{\pi}{4}\right)
        \frac{\mathrm{Ai}}{\sqrt{\mathrm{Ai}^2+\mathrm{Bi}^2}}\left(-2^{1/3}z\right)+ \label{equ1}\\
       & \ \cos\left( \pi x R g\left(\frac{n/x}{R}\right)-\frac{\pi}{4}\right)
       \frac{\mathrm{Bi}}{\sqrt{\mathrm{Ai}^2+\mathrm{Bi}^2}}\left(-2^{1/3}z\right)+E(x)\bigg], \label{equ2}
\end{align}
where
\begin{equation*}
E(x)=O\left(n^{-2/3+2.5\varepsilon}\right).
\end{equation*}
To get the bound of $E$, we used the fact that for large $r$
\begin{equation}
\mathrm{Ai}^2(-r)+\mathrm{Bi}^2(-r) \sim \pi^{-1} r^{-1/2}  \label{case2.5-4}
\end{equation}
(see  \S 2.2 in \cite[p. 395]{olver:1997}) and noticed that $x^{-1}\asymp n^{-1}$.

In order to use the angle sum formula for the sine to simplify \eqref{equ1} and \eqref{equ2}  we define an angle function $\beta:\mathbb{R}\rightarrow (-\infty, 1/2)$ as follows
\begin{equation*}
\beta(z)=\left\{
\begin{array}{ll}
-(m-1)+\frac{1}{\pi}\arctan \frac{\mathrm{Bi}(-2^{1/3}z)}{\mathrm{Ai}(-2^{1/3}z)},  & \textrm{$z\in (\frac{t_{m-1}}{2^{1/3}}, \frac{t_m}{2^{1/3}})$, $m\in\mathbb{N}$,}\\
-(m-1)-\frac{1}{2},  & \textrm{$z=\frac{t_m}{2^{1/3}}$, $m\in\mathbb{N}$,}
\end{array}\right.
\end{equation*}
where $t_0=-\infty$ and $t_m$ ($m\in\mathbb{N}$) is the $m$th zero of the equation $\textrm{Ai}(-x)=0$. Then the terms in brackets in \eqref{equ1} and \eqref{equ2} become
\begin{equation}
\left[\sin\left( \pi x R g\left(\frac{n/x}{R}\right)+\pi \beta(z)-\frac{1}{4}\pi\right)+E(x)\right].\label{case2.5-3}
\end{equation}
Notice that if $z\geq 0$ then \eqref{bound-zeta-} implies
\begin{equation*}
rx g\left(\frac{n}{rx}\right)=\frac{2\sqrt{2}}{3\pi}z^{3/2}+O\left(z^{2.5}n^{-2/3}\right).
\end{equation*}
Define a real-valued function $\psi=\psi(z)$ by
\begin{equation*}
\psi(z)=\left\{ \begin{array}{ll}
\beta(z)+\frac{2\sqrt{2}}{3\pi}z^{3/2}-\frac{1}{4},  & \textrm{$z\geq 0$},\\
\beta(z)-\frac{1}{4},        & \textrm{$z\leq 0$}.
\end{array}\right.
\end{equation*}
By rewriting \eqref{case2.5-3} with this $\psi$ and the function $G$, we get \eqref{case2.5-1} and \eqref{case2.5-5}.

One can easily check, after checking that $\psi'$ is always negative, that $\psi$ does satisfy those properties claimed in the statement of the lemma. For non-positive $z$, $\psi'<0$ follows trivially from the formula
\begin{equation*}
\beta'(z)=-\frac{2^{1/3}/\pi^2}{\left(\mathrm{Ai}^2+\mathrm{Bi}^2\right)(-2^{1/3}z)},
\end{equation*}
in whose calculation we have used the 10.4.10 in \cite[p. 446]{abram:1972}. To prove the inequality for positive $z$, it suffices to show that
\begin{equation*}
\pi z^{1/2} \left(\mathrm{Ai}^2+\mathrm{Bi}^2 \right)(-z)<1 \quad \textrm{for all $z>0$}.
\end{equation*}
This follows from the fact that the left hand side is an increasing function of $z$ (see \S 2.4 in \cite[p. 397]{olver:1997} and \S 7.3 in \cite[p. 342]{olver:1997} or \S 13.74 \cite[P.446]{watson:1966}) and \eqref{case2.5-4}. Since $\psi'(z)\rightarrow 0$ as $|z|\rightarrow \infty$, its image should be a bounded interval.
\end{proof}

\begin{remark}
It follows from the 10.4.78 in \cite[p. 449]{abram:1972} that for large $z$
\begin{equation*}
\psi'(z)=-\frac{5\sqrt{2}}{64\pi}z^{-5/2}+O\left(z^{-11/2}\right).
\end{equation*}
\end{remark}


\begin{lemma} \label{case0}
For all $n\in \mathbb{N}\cup \{0\}$, at any positive zero of $f_n(x)$,
\begin{equation*}
Rx_{n,k}>\sqrt{n^2+\pi^2\left(k-\frac{1}{4} \right)^2}.
\end{equation*}
In particular, $Rx_{n, k}>n$.
\end{lemma}

\begin{proof}
Let $j_{n,k}$ denote the $k$th positive zero of $J_n$. Then
\begin{equation*}
x_{n, k}\geq j_{n,k}/R
\end{equation*}
since for fixed $R$, $n$, and $k$, the zero $x_{n,k}$ (as a function of $r$) is increasing in $r$  (by a similar argument as in the proof of Theorem 2 on \cite[P.222]{Cochran:1966}) and converges to $j_{n,k}/R$ as $r\to 0$ (see \cite[P.38]{Kline:1948}). R. McCann \cite[P.102]{McCann:1977} gives
\begin{equation*}
j_{n,k}>\sqrt{n^2+\pi^2\left(k-\frac{1}{4} \right)^2},
\end{equation*}
hence the desired bound.
\end{proof}


\begin{lemma}\label{case3}
For any $\varepsilon>0$ and all sufficiently large $n$ if $rn/R<rx\leq n-n^{1/3+\varepsilon}$ then
\begin{equation}
f_n(x)=\frac{Y_n(rx)\left(12\pi x G\left(\frac{n}{x} \right)\right)^{1/6}}{\left(\left(Rx\right)^2-n^2\right)^{1/4}}
\left(\mathrm{Ai}\left(-\left(\frac{3\pi}{2} x G\left(\frac{n}{x} \right)\right)^{2/3}\right)+E_4(x)\right),\label{case3-1}
\end{equation}
where $Y_n(rx)<0$ and
\begin{equation}
E_4(x)=O\left(n^{-4/3}\max\left\{1, \left(x G\left(\frac{n}{x} \right)\right)^{1/6}\right\}\right).\label{case3-3}
\end{equation}
If we further assume that $x G(n/x)>1$, then
\begin{equation}
f_n(x)=\sqrt{\frac{2}{\pi}} \frac{Y_n(rx)}{\left(\left(Rx\right)^2-n^2\right)^{1/4}}
\left(\sin\left( \pi x G\left(\frac{n}{x}\right)+\frac{\pi}{4}\right)+E_5(x)\right),\label{case3-2}
\end{equation}
where
\begin{equation}
E_5(x)=O\left(\left(x G\left(\frac{n}{x}\right)\right)^{-1} \right). \label{case3-5}
\end{equation}
\end{lemma}

\begin{remark}
Comparing the asymptotics \eqref{case3-2} with \eqref{case2.5-1} at the same point $x=r^{-1}(n-n^{1/3+\varepsilon})$, we notice that $E_5(x)=O(n^{-1})$ is a better bound than $E_3(x)=O(n^{-2/3+2.5\varepsilon})$. This difference is due to the different methods used to prove these two lemmas. In fact, if we expand $Y_n(rx)$ at $x=r^{-1}(n-n^{1/3+\varepsilon})$ by Lemma \ref{9.3.4analogue}, then $E_5(x)$ becomes $E_3(x)$ and \eqref{case3-2} is consistent with \eqref{case2.5-1}.
\end{remark}

\begin{proof}[Proof of Lemma \ref{case3}]
As in the proof of Lemma \ref{case222}, we denote
\begin{equation*}
Rx=n z_{R} \quad \textrm{and}\quad  rx=n z_{r}.
\end{equation*}
Since $1<z_R< R/r$ the $\zeta_R=\zeta(z_R)$, determined by \eqref{def-zeta1}, is negative such that
\begin{equation*}
0<(-\zeta_R)^{3/2}\ll 1.
\end{equation*}
Meanwhile, since $r/R< z_r\leq 1-n^{-2/3+\varepsilon}$ the $\zeta_r=\zeta(z_r)$, determined by \eqref{def-zeta2}, is positive such that
\begin{equation*}
n^{-1+3\varepsilon/2}\ll \zeta_r^{3/2}\ll 1
\end{equation*}
whenever $n$ is sufficiently large.

With the estimate
\begin{equation*}
n^{\varepsilon}\ll n^{2/3}\zeta_r\ll n^{2/3},
\end{equation*}
applying Olver's asymptotic expansions \eqref{jnuse111} and \eqref{ynuse111} and asymptotics for the Airy functions (the 10.4.59, 10.4.61, 10.4.63, and 10.4.66 in \cite{abram:1972}) yields
\begin{equation*}
J_n(rx)=\left(2\pi\right)^{-1/2}\left(n^2-(rx)^2\right)^{-1/4}e^{-\frac{2}{3}n\zeta_r^{3/2}} \left(1+O\left(n^{-1}\zeta_r^{-3/2}\right)\right)
\end{equation*}
and
\begin{equation*}
Y_n(rx)=-\left(2/\pi\right)^{1/2}\left(n^2-(rx)^2\right)^{-1/4}e^{\frac{2}{3}n\zeta_r^{3/2}} \left(1+O\left(n^{-1}\zeta_r^{-3/2}\right)\right).
\end{equation*}
Hence $Y_n(rx)$ is always negative and
\begin{equation*}
\frac{J_n(rx)}{Y_n(rx)}=-\frac{1}{2}e^{-\frac{4}{3}n\zeta_r^{3/2}}\left(1+O\left(n^{-1}\zeta_r^{-3/2}\right)\right)
=O\left(e^{-n^{\varepsilon}}\right).
\end{equation*}
Therefore
\begin{equation}
f_n(x)=Y_n(rx)\left(J_n(n z_{R})+Y_n(n z_{R})O\left(e^{-n^{\varepsilon}}\right)\right).\label{case3-4}
\end{equation}

Notice that
\begin{equation*}
0<n^{2/3}(-\zeta_R)\ll n^{2/3}.
\end{equation*}
We discuss in two cases depending on whether $n^{2/3}(-\zeta_R)$ is large or small. Applying Olver's asymptotic expansions to $J_n(n z_{R})$ and $Y_n(n z_{R})$ in \eqref{case3-4} and bounds for the Airy functions yields the formula \eqref{case3-1} with
\begin{equation*}
E_4(x)=\left\{
\begin{array}{ll}
O\left(n^{-4/3}\left(n^{2/3}|\zeta_R|\right)^{1/4}\right),  & \textrm{if $n^{2/3}(-\zeta_R)\geq 1$,}\\
O\left(n^{-4/3}\right),  & \textrm{if $n^{2/3}(-\zeta_R)< 1$,}
\end{array}\right.
\end{equation*}
hence the bound \eqref{case3-3}, where we have used the fact
\begin{equation*}
n^{2/3}\left(-\zeta_R\right)=\left(\frac{3\pi}{2} x G\left(\frac{n}{x} \right)\right)^{2/3}.
\end{equation*}

It follows easily from \eqref{case3-1} and \eqref{case3-3} to get \eqref{case3-2} and \eqref{case3-5} by using the well-known asymptotics of $\mathrm{Ai}(-r)$.
\end{proof}


We can now collect all previous lemmas and give a description of zeros of $f_n$ for large $n$.

\begin{theorem} \label{thm111}
There exists a constant $c\in (0,1)$ such that for any $\varepsilon>0$ and all sufficiently large $n$ the positive zeros of $f_n$, $\{x_{n,k}\}_{k=1}^{\infty}$, satisfy the following:
\begin{enumerate}
\item if $rx_{n,k}\geq (1+c)n$ then
\begin{equation}
x_{n,k} G\left(\frac{n}{x_{n,k}}\right)=k+O\left(x_{n,k}^{-1}\right);\label{thm111-1}
\end{equation}

\item if $n+n^{1/3+\varepsilon}\leq rx_{n,k}<(1+c)n$ then
\begin{equation}
x_{n,k} G\left(\frac{n}{x_{n,k}}\right)=k+O\left(z_{n,k}^{-3/2}\right) \label{thm111-2}
\end{equation}
with $z_{n,k}$ determined by the equation $rx_{n,k}=n+z_{n,k} n^{1/3}$;

\item if $n-n^{1/3+\varepsilon}< rx_{n,k}< n+n^{1/3+\varepsilon}$ then
\begin{equation}
x_{n,k}G\left(\frac{n}{x_{n,k}}\right)=k-\psi\left(z_{n,k}\right)+O\left(n^{-2/3+2.5\varepsilon}\right),\label{thm111-3}
\end{equation}
where $\psi$ is the function appearing in Lemma \ref{case2.5} and $z_{n,k}$ is defined as above;

\item if $rx_{n,k}\leq n-n^{1/3+\varepsilon}$ then
\begin{equation}
x_{n,k} G\left(\frac{n}{x_{n,k}}\right)=k-\frac{1}{4}+E_{n,k},\label{thm111-4}
\end{equation}
where
\begin{equation*}
|E_{n,k}|<\min\left\{\frac{3}{8}, O\left(\left(x_{n,k} G\left(\frac{n}{x_{n,k}}\right)\right)^{-1}\right)\right\}.
\end{equation*}
\end{enumerate}
\end{theorem}

\begin{proof}
The rough idea of this proof is to apply to $f_n(x)$ the intermediate value theorem in the interval $(0, (s+1/2)\pi/(R-r))$ for any sufficiently large integer $s$ and then J. Cochran's result of the number of zeros within such an interval (see \cite{cochran:1964}).

We will study zeros of $f_n$ only in $(n/R, (s+1/2)\pi/(R-r))$ for any sufficiently large integer $s>n^3$ since Lemma \ref{case0} tells us that there is no zeros $\leq n/R$. Inspired by the asymptotics obtained in this section, the study will be done via discussing the values of $h_n(x):=x G(n/x)$\footnote{This notation $h_n$ will be used through the rest of this section.} on the chosen interval (see Figure \ref{hn} for an example of the graph of $h_n$).
\begin{figure} [ht]
\centering
\includegraphics[width=0.6\textwidth]{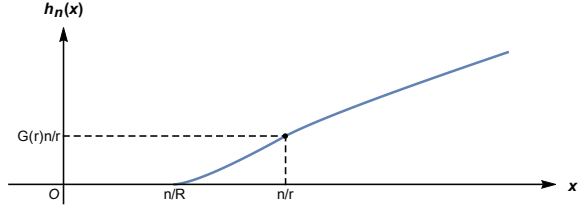} 
\caption{The graph of $h_n$ with $n=30$, $R=2$ and $r=1$.}
\label{hn}
\end{figure}
We observe that $h_n: [n/R, \infty)\rightarrow [0, \infty)$ is a continuous and strictly increasing function that maps $(n/R, (s+1/2)\pi/(R-r))$ onto $(0, s+1/2+O(n^{-1}))$. Therefore for each integer $1\leq k\leq s$ there exists an interval $(a_k, b_k)\subset (n/R, (s+1/2)\pi/(R-r))$ such that $h_n$ maps $(a_k, b_k)$ to $(k-3/8, k+1/8)$ bijectively. It is obvious that these intervals $(a_k, b_k)$'s are disjointly located one by one as $k$ increases.

We claim that if $n$ is sufficiently large then for each $1\leq k\leq s$
\begin{equation}
f_n(a_k)f_n(b_k)<0.\label{IVT-condition}
\end{equation}
If this is true, the intermediate value theorem ensures the existence of at least one zero of $f_n$ in each $(a_k, b_k)$. Recall that there are exactly $s$ zeros of $f_n$ in $(0, (s+1/2)\pi/(R-r))$ (see \cite{cochran:1964}). Hence there exists one and only one zero in each $(a_k, b_k)$, which must be $x_{n,k}$ by definition.

To verify \eqref{IVT-condition} we take advantage of the asymptotics \eqref{case111-1},  \eqref{case222-1}, \eqref{case2.5-1}, \eqref{case3-2} and \eqref{case3-1}. In all cases except the last one (when $x G(n/x)<C$ for a sufficiently large $C$), the verification is easy if we notice that
\begin{equation*}
h_n(a_k)+\delta_1 \in \left[k-\frac{3}{8}, k-\frac{1}{8}\right]
\quad \textrm{and} \quad h_n(b_k)+\delta_2\in \left[k+\frac{1}{8}, k+\frac{3}{8}\right]
\end{equation*}
for any $0\leq \delta_1, \delta_2\leq 1/4$.  In the last case when $x G(n/x)<C$ we use the asymptotics \eqref{case3-1}. The sign of $f_n$ depends on that of
\begin{equation}
\mathrm{Ai}\left(-\left(3\pi h_n(x)/2 \right)^{2/3}\right)+O_C\left(n^{-4/3}\right). \label{theorem-1}
\end{equation}
As in the proof of Lemma \ref{case2.5}, we denote by $t_k$ ($k\in\mathbb{N}$) the $k$th zero of the equation $\textrm{Ai}(-x)=0$. \cite[P.405]{olver:1997} gives that
\begin{equation*}
t_k=\left[\frac{3\pi}{2}\left(k-\frac{1}{4}+\alpha'_k\right)\right]^{2/3}
\end{equation*}
with a crude estimate $|\alpha'_k|<0.11$. Thus
\begin{align}
t_k\in &\left(\left[\frac{3\pi}{2}\left(k-0.36\right)\right]^{2/3}, \left[\frac{3\pi}{2}\left(k-0.14\right)\right]^{2/3} \right) \nonumber \\
&\subsetneq \left(\left[\frac{3\pi}{2}h_n(a_k)\right]^{2/3}, \left[\frac{3\pi}{2}h_n(b_k)\right]^{2/3} \right). \label{theorem-2}
\end{align}
Since $\textrm{Ai}(-x)$ oscillates around zero for positive $x$ and the intervals in \eqref{theorem-2} are disjoint for different $k$'s, the signs of \eqref{theorem-1} at $x=a_k$ and $b_k$ must be opposite whenever $n$ is sufficiently large, which in turn gives \eqref{IVT-condition} in the last case.

We are now able to finish the proof of the theorem. For each zero $x_{n,k}$, $f_n(x_{n,k})=0$. If $h_n(x_{n,k})\geq C$ we apply to the left hand side either \eqref{case111-1},  \eqref{case222-1}, \eqref{case2.5-1} or \eqref{case3-2}, and conclude that the factor involving the sine function and $h_n(x_{n,k})$ has to be zero. Since $h_n(x_{n,k})+\delta$ is always in the interval $[k-3/8, k+3/8]$ for any $0\leq \delta\leq 1/4$, applying the arcsine function immediately yields the desired asymptotics. If $h_n(x_{n,k})< C$ we use the fact that $h_n(x_{n,k})\in (h_n(a_k), h_n(b_k))$ to get a crude estimate $|h_n(x_{n,k})-(k-1/4)|<3/8$.
\end{proof}

%
%


For small $n$ we have the following.

\begin{theorem} \label{thm222}
For any $N\in\mathbb{N}$ there exists a constant $K>0$ such that if $0\leq n\leq N$ and $k\geq K$ then the positive zero $x_{n,k}$ of $f_n$ satisfies
\begin{equation}
x_{n,k} G\left(\frac{n}{x_{n,k}}\right)=k+O\left(x_{n,k}^{-1}\right). \label{thm222-2}
\end{equation}
\end{theorem}

\begin{proof}
If $0\leq n\leq N$ and $x>C_N$ for a sufficiently large constant $C_N$ then Lemma \ref{case111} (with $c=1$) gives a factorization \eqref{case111-1} of $f_n$ with $|E_1(x)|<1/100$. By using such a factorization we study $f_n$ on the interval
\begin{equation}
\left[\frac{(k-1/2)\pi}{R-r}, \frac{(k+1/2)\pi}{R-r}\right), \label{thm222-1}
\end{equation}
which is a subset of $(C_N, \infty)$ if $k$ is sufficiently large. As in the proof of Theorem \ref{thm111}, we then study the function $h_n(x)=x G(n/x)$ on a subinterval of \eqref{thm222-1}, denoted by $(a_k, b_k)$, with $h_n((a_k, b_k))=(k-3/8, k+1/8)$. Such a subset indeed exists if $k$ is sufficiently large since $h_n((k\pm 1/2)\pi/(R-r))=k\pm 1/2+O(N^2/k)$.

It is easy to see that $f_n(a_k)f_n(b_k)<0$. By the intermediate value theorem there exists at least one zero of $f_n$ in $(a_k, b_k)$, which must be $x_{n,k}$ since there exists exactly one zero  in the interval \eqref{thm222-1} if $k$ is sufficiently large (due to the fact that there are exactly $s$ zeros of $f_n$ in $(0, (s+1/2)\pi/(R-r))$ for sufficiently large integer $s$ (see \cite{cochran:1964})). Thus
\begin{equation*}
\sin\left( \pi h_n(x_{n,k})\right)+E_1(x_{n,k})=0.
\end{equation*}
Applying the arcsine function  yields the desired result.
\end{proof}


\begin{corollary}\label{cor1}
Given any sufficiently large integer $n$ and $0<\sigma<R$, for all $x_{n,k}$'s that are greater than $n/\sigma$ we have
\begin{equation*}
1 \ll x_{n,k+1}-x_{n,k}\ll_{\sigma} 1.
\end{equation*}
Furthermore, if $0<\sigma\leq r$ then the dependence of the implicit constant on $\sigma$ can be removed.

For any $N\in\mathbb{N}$ if $0\leq n\leq N$ and $k$ is sufficiently large then
\begin{equation*}
x_{n,k+1}-x_{n,k}\asymp 1.
\end{equation*}
\end{corollary}

\begin{proof}
If $n\geq 1$, a straightforward computation shows that if $x\geq n/r$ then
\begin{equation*}
h_n'(x)=\frac{\left(R^2-r^2\right)/\pi}{\sqrt{R^2-(n/x)^2}+\sqrt{r^2-(n/x)^2}} \in \left(\frac{R-r}{\pi}, \frac{\sqrt{R^2-r^2}}{\pi} \right];
\end{equation*}
if $n/R\leq x\leq n/r$ then
\begin{equation*}
h_n'(x)=\frac{1}{\pi}\sqrt{R^2-(n/x)^2}\in \left[0, \frac{1}{\pi}\sqrt{R^2-r^2} \right].
\end{equation*}
For all sufficiently large $n$ the desired results follow from Theorem \ref{thm111}, the mean value theorem and the above first derivatives. For any  $N\in\mathbb{N}$  and $1\leq n\leq N$, we observe that if $k$ is sufficiently large (depending on $N$) then $x_{n,k}>n/r$. We can then derive the desired result similarly with Theorem \ref{thm111} replaced by Theorem \ref{thm222}.

The case $n=0$ follows trivially from Theorem \ref{thm222}.
\end{proof}

\begin{corollary}\label{cor2}
The error terms in both \eqref{thm111-1} and \eqref{thm222-2}, in \eqref{thm111-2}, in \eqref{thm111-3} and in \eqref{thm111-4} are of size
\begin{equation*}
O\left(\frac{1}{n+k}\right),\quad O\left(\frac{n^{1/2}}{\left(k-\frac{G(r)}{r}n\right)^{3/2}}\right), \quad O\left(n^{-2/3+2.5\varepsilon}\right) \quad \textrm{and}\quad O\left(\frac{1}{k}\right)
\end{equation*}
respectively.
\end{corollary}

\begin{remark}\label{cor2-3}
These bounds are all as small as we want if we choose $n$ or $k$ properly large. It is quite obvious to observe this except (perhaps) for the second bound. As to that, we just need to notice \eqref{cor2-2} below and the corresponding range of $z_{n,k}$, namely $n^{\varepsilon}\leq z_{n,k}<cn^{2/3}$.
\end{remark}

\begin{proof}[Proof of Corollary \ref{cor2}]
For \eqref{thm111-1} and \eqref{thm222-2} the desired bound follows easily from Lemma \ref{case0}. For \eqref{thm111-3} and \eqref{thm111-4} the bounds can be obtained directly from the asymptotics themselves.

We will focus on the error term in \eqref{thm111-2} below and prove that
\begin{equation}
z_{n,k}\asymp n^{-1/3}\left(k-\frac{G(r)}{r}n\right).\label{cor2-2}
\end{equation}
Let $k_0, k\in \mathbb{N}$ be such that
\begin{equation}
rx_{n,k_0-1}<n\leq rx_{n,k_0} \label{cor2-1}
\end{equation}
and
\begin{equation*}
n+n^{1/3+\varepsilon}\leq rx_{n,k}<(1+c)n.
\end{equation*}
Hence, by Corollary \ref{cor1}, $k-k_0\asymp x_{n,k}-x_{n,k_0} \gg n^{1/3+\varepsilon}$ which is much greater than $1$. Since $z_{n,k_0}\geq 0>z_{n,k_0-1}$ we have
\begin{equation*}
z_{n,k}\geq z_{n,k}-z_{n,k_0}=rn^{-1/3}\left(x_{n,k}-x_{n,k_0} \right)\asymp n^{-1/3}\left(k-k_0 \right)
\end{equation*}
and
\begin{equation*}
z_{n,k}< z_{n,k}-z_{n,k_0-1}=rn^{-1/3}\left(x_{n,k}-x_{n,k_0-1} \right)\asymp n^{-1/3}\left(k-k_0 \right).
\end{equation*}
By using \eqref{thm111-3}, \eqref{cor2-1} and the monotonicity of $h_n$ we have
\begin{equation*}
k-k_0\asymp k-\frac{G(r)}{r}n.
\end{equation*}
We therefore obtain \eqref{cor2-2}.
\end{proof}

Let $F: [0, \infty)\times [0, \infty)\setminus \{O\}\rightarrow \mathbb{R}$ be the function homogeneous of degree $1$ which satisfies $F\equiv1$ on the graph of $G$. By implicit differentiation, we have
\begin{equation*}
\partial_y F(x,y)=\frac{1}{(t, G(t))\cdot (-G'(t),1)}
\end{equation*}
and
\begin{equation*}
\partial_x F(x,y)=\frac{-G'(t)}{(t, G(t))\cdot (-G'(t),1)},
\end{equation*}
where $0\leq t<R$ is determined by $ty=G(t)x$, that is, $(t, G(t))$ is the intersection point of the graph of $G$ and the line segment connecting the origin $O$ and the point $(x, y)$. Analyzing the sizes of the above derivatives yields

\begin{lemma}\label{derivativeF}
Following the above notations, we have that if $R-c\leq t<R$ for a sufficiently small constant $c>0$ then
\begin{equation*}
0<\partial_y F(x,y)\asymp x^{1/3}y^{-1/3},
\end{equation*}
otherwise
\begin{equation*}
0<\partial_y F(x,y)\asymp_c 1.
\end{equation*}

We also have $0\leq \partial_x F(x,y)\ll 1$. In particular,  if $0<c'\leq t<R$ then
\begin{equation*}
0< \partial_x F(x,y)\asymp_{c'} 1.
\end{equation*}
\end{lemma}

By  Theorem \ref{thm111} and \ref{thm222}, Corollary \ref{cor2} and Lemma \ref{derivativeF}, we have the following approximations of zeros.

\begin{corollary}\label{approximation}
There exists a constant $c\in (0,1)$ such that for any $\varepsilon>0$ there exists a $N\in\mathbb{N}$ such that if $n>N$ then the positive zeros of $f_n$, $\{x_{n,k}\}_{k=1}^{\infty}$, satisfy
\begin{equation}
x_{n, k}=F(n,k-\tau_{n,k})+R_{n,k},\label{approximation1}
\end{equation}
where
\begin{equation}
\tau_{n,k}=\left\{
            \begin{array}{ll}
            0,                         & \textrm{if $rx_{n,k}\geq n+n^{1/3+\varepsilon}$,}\\
            \psi\left(z_{n,k}\right),  & \textrm{if $n-n^{1/3+\varepsilon}< rx_{n,k}< n+n^{1/3+\varepsilon}$,}\\
            1/4,                       & \textrm{if $rx_{n,k}\leq n-n^{1/3+\varepsilon}$,}
            \end{array}
           \right. \label{translation}
\end{equation}
where $\psi$ is the function appearing in Lemma \ref{case2.5} with $z_{n,k}$ determined by the equation $rx_{n,k}=n+z_{n,k} n^{1/3}$, and
\begin{equation*}
R_{n,k}\!=\!\left\{
            \begin{array}{ll}
            \!O\left((n+k)^{-1}\right),                                         & \textrm{if $rx_{n,k}\geq (1+c)n$,}\\
           \!O\left(n^{1/2}\left(k-\frac{G(r)}{r}n\right)^{-3/2}\right),  & \textrm{if $n+n^{1/3+\varepsilon}\leq rx_{n,k}<(1+c)n$,}\\
            \!O\left(n^{-2/3+2.5\varepsilon}\right),                       & \textrm{if $n-n^{1/3+\varepsilon}< rx_{n,k}< n+n^{1/3+\varepsilon}$,}\\
            \!O\left(n^{1/3}k^{-4/3}\right),                                   &\textrm{if $rx_{n,k}\leq n-n^{1/3+\varepsilon}$.}
            \end{array}
        \right.
\end{equation*}

If $0\leq n\leq N$ there exists a $K\in\mathbb{N}$ such that if $k>K$ then \eqref{approximation1} holds with
\begin{equation}
\tau_{n,k}=\left\{
            \begin{array}{ll}
            0,                         & \textrm{if $k>K$,}\\
            1/4,                       & \textrm{if $1\leq k\leq K$,}
            \end{array}
           \right. \label{translation2}\end{equation}
\footnote{The definition of $\tau_{n,k}$ for $1\leq k\leq K$ is irrelevant here, however we define it anyway for the discussion in the next section.}and
\begin{equation*}
R_{n,k}=O\left((n+k)^{-1}\right).
\end{equation*}
\end{corollary}

\begin{proof}
If $rx_{n,k}>n-n^{1/3+\varepsilon}$ then $x_{n,k}>2n/(R+r)$ for sufficiently large $n$. By using \eqref{thm111-1}--\eqref{thm111-3} and the monotonicity of $h_n$, we have
\begin{equation*}
\frac{k}{n}\geq \frac{h_n(x_{n,k})}{2n}\geq \frac{1}{R+r}G\left(\frac{R+r}{2}\right),
\end{equation*}
which, by Lemma \ref{derivativeF}, ensures that $\partial_y F(n, k)\asymp 1$. The \eqref{approximation1} with $rx_{n,k}>n-n^{1/3+\varepsilon}$ then follows from \eqref{thm111-1}--\eqref{thm111-3}, the mean value theorem and Corollary \ref{cor2}.

If $rx_{n,k}\leq n-n^{1/3+\varepsilon}$ then $x_{n,k}<n/r$. We argue as above to get that
\begin{equation*}
\frac{k}{n}=\frac{h_n(x_{n,k})+O(1)}{n}\ll 1.
\end{equation*}
By Lemma \ref{derivativeF}, if $k/n$ is sufficiently small then $\partial_y F(n, k)\asymp n^{1/3}k^{-1/3}$; otherwise $k/n \asymp 1$ which also ensures that $\partial_y F(n, k)\asymp n^{1/3}k^{-1/3}$. The \eqref{approximation1} with $rx_{n,k}\leq n-n^{1/3+\varepsilon}$ thus follows from \eqref{thm111-4}, the mean value theorem and Corollary \ref{cor2}.

At last, the case $0\leq n\leq N$ follows easily from Theorem \ref{thm222}, Corollary \ref{cor2} and Lemma \ref{derivativeF}.
\end{proof}


\section{Spectrum counting to lattice counting}\label{reduction-sec}

Consider the Dirichlet Laplacian operator $\triangle$ on the planar annulus $\mathscr{D}$. Using the standard separation of variables, we know that its spectrum contains exactly the numbers $x_{n,k}^2$, $n\in\mathbb{N}\cup\{0\}$, $k\in\mathbb{N}$, defined at the beginning of Section \ref{zeros}. We also know that in the spectrum each $x_{n,k}$ appears twice for every fixed $n\in\mathbb{N}$ and only once if $n=0$. If we define $x_{n,k}=x_{-n,k}$ for any negative integer $n$, then the spectrum counting function $\mathscr{N}_{\mathscr{D}}(\mu)$ defined by \eqref{e-counting} becomes
\begin{equation*}
\mathscr{N}_{\mathscr{D}}(\mu)=\#\left\{(n, k)\in\mathbb{Z}\times \mathbb{N} : x_{n,k}\leq \mu \right\}.
\end{equation*}

Recall that we define in Corollary \ref{approximation} (with the $c$ and $\varepsilon$ appearing there fixed) the amount of translation $\tau_{n,k}$ for $n\in \mathbb{N}\cup \{0\}$, $k\in \mathbb{N}$, namely, \eqref{translation} and \eqref{translation2}. We now extend its definition to $\mathbb{Z}^2$ by letting $\tau_{n,k}$ be $\tau_{-n,k}$ if $n<0$ and $1/4$ if $k\leq 0$. In view of the multiplicity of the spectrum and Corollary \ref{approximation}, each $x_{n,k}$, $n\in\mathbb{Z}$, corresponds to a unique point $(n,k-\tau_{n,k})$.

Denote by $\mathcal{D}$ the closed domain symmetric about the $y$-axis and in the first quadrant bounded by the graph of $G$ and the $x$-axis. See the shaded area in Figure \ref{SymmH}. Define a lattice counting function $\mathcal{N}_{\mathcal{D}}(\mu)$ by
\begin{equation*}
\mathcal{N}_{\mathcal{D}}(\mu)=\#\left(\mu\mathcal{D}\cap \left\{(n,k-\tau_{n,k}) : (n,k)\in\mathbb{Z}^2\right\}\right), \quad \mu>2.
\end{equation*}
\begin{figure}[ht]
\centering
\includegraphics[width=0.6\textwidth]{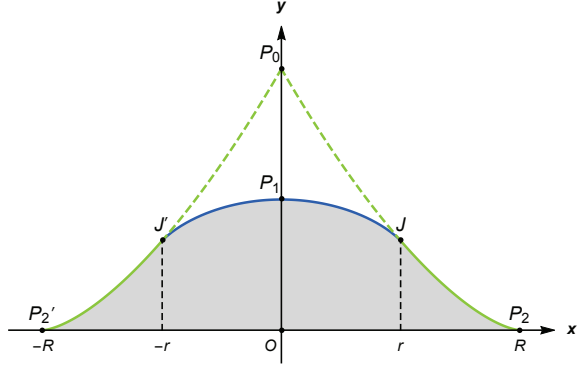} 
\caption{The symmetric domain $\mathcal{D}$.}
\label{SymmH}
\end{figure}

Then one can transfer the spectrum counting problem to a lattice counting problem via the following result. In its proof we essentially follow the treatment for ``the boundary parts'' in \cite[Theorem 3.1]{colin:2011}.

\begin{proposition}\label{difference1}
There exists a constant $C>0$ such that
\begin{equation}
\left|\mathscr{N}_{\mathscr{D}}(\mu)-\mathcal{N}_{\mathcal{D}}(\mu) \right|\leq \mathcal{N}_{\mathcal{D}}\left(\mu+C\mu^{-0.4}\right)-
\mathcal{N}_{\mathcal{D}}\left(\mu-C\mu^{-0.4}\right)+O\left(\mu^{0.6} \right). \label{222}
\end{equation}
\end{proposition}

\begin{proof}
To study $\mathscr{N}_{\mathscr{D}}(\mu)$ we would like to use the approximations of $x_{n, k}$'s given by Corollary \ref{approximation}. We need to assume that $\max\{|n|,k\}$ is sufficiently large, however, we will not emphasize this explicitly in the following argument. This treatment will not cause any problem; after all, it will produce at most an $O(1)$ error, which is much less than the error term $O(\mu^{0.6})$ in \eqref{222}.

For $k\in\mathbb{N}$, let
\begin{equation*}
\mathscr{N}_k(\mu):=\#\left\{n\in\mathbb{N} : x_{n,k}\leq \mu \right\}=\#\left\{n\in\mathbb{N} : F(n,k-\tau_{n,k})+R_{n,k}\leq \mu\right\}
\end{equation*}
and
\begin{equation*}
\mathcal{N}_k(\mu):=\#\left\{n\in\mathbb{N} : (n,k-\tau_{n,k})\in\mu\mathcal{D}\right\}=\#\left\{n\in\mathbb{N} : F(n,k-\tau_{n,k})\leq\mu\right\}.
\end{equation*}
Then
\begin{align}
\Delta_k(\mu):&=|\mathscr{N}_k(\mu)-\mathcal{N}_k(\mu)| \nonumber \\
            &\leq \#\left\{n\in\mathbb{N} : \mu-|R_{n,k}|\leq F(n,k-\tau_{n,k})\leq \mu+|R_{n,k}|\right\}. \label{999}
\end{align}
Hence we only need to consider points $(n,k-\tau_{n,k})$ satisfying $F(n,k-\tau_{n,k})=\mu+O(1)$.

We next use Lemma \ref{derivativeF} and Corollary \ref{approximation} to obtain bounds of $\Delta_k(\mu)$.   We discuss in several cases depending on the size of $k$.

If $1\leq k\leq \mu^{1/4}$ then
\begin{equation}
\Delta_k(\mu)\ll \mu^{1/3}k^{-4/3}.\label{444}
\end{equation}
Indeed, in this case we have $n\asymp \mu$. This, together with Theorem \ref{thm111}, leads to
\begin{equation*}
\frac{G(n/x_{n,k})}{n/x_{n,k}}=\frac{k+O(1)}{n}\ll \mu^{-3/4},
\end{equation*}
which implies that $n/x_{n,k}$ is close to $R$ and thus $rx_{n,k}\leq n-n^{1/3+\varepsilon}$. Therefore $R_{n,k}=O(n^{1/3}k^{-4/3})$. Using the estimate of $\partial_x F$ we get \eqref{444}.

If $\mu^{1/4}< k\leq \mu^{4/7}$ then a similar argument as above shows that $R_{n,k}=O(n^{1/3}k^{-4/3})=O(1)$ and
\begin{equation*}
\Delta_k(\mu)\ll 1. \footnote{In this case $R_{n,k}$ may be much smaller than $1$, but there may exist one element in the set in \eqref{999}. Hence we may only use the trivial bound $O(1)$ for $\Delta_k(\mu)$.}
\end{equation*}

If $\mu^{4/7}< k\leq G(r)\mu-C_1$ for a sufficiently large constant $C_1$ (to be determined below) then
\begin{equation}
\Delta_k(\mu)\leq \mathcal{N}_k(\mu+C\mu^{-3/7})-\mathcal{N}_k(\mu-C\mu^{-3/7})\label{666}
\end{equation}
for some constant $C$. Indeed, let us fix arbitrarily an element $n$ belonging to the set in \eqref{999}, hence the point $(n,k-\tau_{n,k})$ is contained in a tubular neighborhood of $\mu\partial \mathcal{D}$ of width much less than $1$ (see Remark \ref{cor2-3}). Since $G'$ is continuous at the point $x=r$ and $G'(r)\in (-1/2, 0)$, as $\mu\rightarrow \infty$ the tubular neighborhood (mentioned above) between $y=G(r)\mu$ and $y=G(r)\mu-C_1$ is close to a parallelogram. A simple geometric argument ensures that if $C_1$ is a sufficiently large constant then $n\geq r\mu$. As a result,
\begin{equation}
\frac{k}{n}\leq \frac{G(r)}{r}-\frac{C_1}{n}.\label{555}
\end{equation}
On the other hand side we observe, as a consequence of Theorem \ref{thm111} and the monotonicity of $G$, that if $rx_{n,k}\geq n+n^{1/3+\varepsilon}$ then
\begin{equation*}
\frac{k}{n}=\frac{G(n/x_{n,k})}{n/x_{n,k}}+O\left(n^{-1-\frac{3}{2}\varepsilon} \right)>\frac{G(r)}{r}+O\left(n^{-1-\frac{3}{2}\varepsilon}\right),
\end{equation*}
which contradicts with \eqref{555}. Therefore $rx_{n,k}< n+n^{1/3+\varepsilon}$ and $R_{n,k}$ can only be either $O(n^{-2/3+2.5\varepsilon})$ or $O(n^{1/3}k^{-4/3})$, both of which are of size $O(\mu^{-3/7})$ since $n\asymp \mu$. We then readily get \eqref{666}.

If $G(r)\mu-C_1< k\leq G(r)\mu+\mu^{0.6}$ then the trivial estimate $R_{n,k}=O(1)$ yields that
\begin{equation*}
\Delta_k(\mu)\ll 1.
\end{equation*}

If $k>G(r)\mu+\mu^{0.6}$ then
\begin{equation}
\Delta_k(\mu)\leq \mathcal{N}_k(\mu+C\mu^{-0.4})-\mathcal{N}_k(\mu-C\mu^{-0.4})\label{888}
\end{equation}
for some constant $C$. Since the proof is almost the same as that of \eqref{666}, let us be brief. We still fix arbitrarily an element $n$ belonging to the set in \eqref{999}. A geometric argument shows that $n<r\mu$. Thus
\begin{equation*}
\frac{k}{n}> \frac{G(r)}{r}+\frac{\mu^{0.6}}{n}.
\end{equation*}
However, if $rx_{n,k}\leq n-n^{1/3+\varepsilon}$ then
\begin{equation*}
\frac{k}{n}=\frac{G(n/x_{n,k})}{n/x_{n,k}}+\frac{1/4+O(1)}{n}<\frac{G(r)}{r}+\frac{O(1)}{n},
\end{equation*}
which is impossible. Hence $rx_{n,k}>n-n^{1/3+\varepsilon}$ and $R_{n,k}$ can be in the form of $O\left((n+k)^{-1}\right)$, $O\left(n^{1/2}\left(k-\frac{G(r)}{r}n\right)^{-3/2}\right)$ or $O(n^{-2/3+2.5\varepsilon})$. In fact, we further observe that if $n/\mu$ is sufficiently small then $k/n$ is sufficiently large and $R_{n,k}$ must be $O\left((n+k)^{-1}\right)$, as a consequence of Theorem \ref{thm111} and \ref{thm222}. To conclude the proof of \eqref{888}, we only need to notice that no matter in which form the $R_{n,k}$ is, it is always of size $O(\mu^{-0.4})$.

If $n=0$, by using exactly the same argument as above we get
\begin{align}
&\left|\#\left\{k\in\mathbb{N} : x_{0,k}\leq \mu \right\}-\#\left\{k\in\mathbb{N} : (0,k-\tau_{0,k})\in\mu\mathcal{D}\right\}\right| \label{333}\\
&\quad \leq \#\left\{k\in\mathbb{N} : (0,k-\tau_{0,k})\in \left(\mu+C\mu^{-1}\right)\mathcal{D}\setminus \left(\mu-C\mu^{-1}\right)\mathcal{D}  \right\} \nonumber
\end{align}
for some constant $C>0$.

To conclude, summing the above bounds of $\Delta_k(\mu)$ over $k\in\mathbb{N}$ and using the symmetry between positive and negative $n$'s and the bound \eqref{333} yields the desired inequality.
\end{proof}


$\mathcal{N}_{\mathcal{D}}(\mu)$ counts the number of lattice points (under various translations) in $\mu\mathcal{D}$. This feature brings us some obstacles in its estimation. To overcome this difficulty we move every point $(n, k-\tau_{n,k})$ to $(n, k-1/4)$ to obtain an uniformity in translation,  and then study the relatively standard lattice counting function
\begin{equation}\label{lattice-pro2}
\mathcal{N}_{\mathcal{D}}^{u}(\mu)=\#\left(\mu\mathcal{D}\cap \left\{(n,k-1/4) : (n,k)\in\mathbb{Z}^2\right\}\right), \quad \mu>2.
\end{equation}
Here the superscript ``u'' represents the uniformity in translation. Of course such a transformation from $\mathcal{N}_{\mathcal{D}}(\mu)$ to $\mathcal{N}_{\mathcal{D}}^{u}(\mu)$ will cause a difference. To quantify that we need to count the number of lattice points in a band of length $r\mu$ and width $1/4$. (This will be clear in the proof of the next proposition.)

For $0<L\leq R\mu$ let us define a band on $[0, L]$ by
\begin{equation*}
\mathcal{B}_{L}=\left\{(x,y)\in\mathbb{R}^2 : 0\leq x\leq L,   \,    \mu G\left(\frac{x}{\mu} \right)< y\leq \mu G\left(\frac{x}{\mu} \right)+\frac{1}{4} \right\}
\end{equation*}
and the number of $\mathbb{Z}^2$ in the band $\mathcal{B}_{r\mu}$ by
\begin{equation}
\# \left(\mathcal{B}_{r\mu}\cap\mathbb{Z}^2 \right)=\frac{1}{4}r\mu+\mathcal{E}(\mu). \label{error-in-band}
\end{equation}

One would expect the error term $\mathcal{E}(\mu)$ to be much smaller than the linear term $r\mu/4$ since heuristically the number of lattice points inside a large planar domain is asymptotically equal to the area of the domain with an error term that is not too bad if the curvature involved does not vanish. We will estimate $\mathcal{E}(\mu)$ in the next section.

With $\mathcal{E}(\mu)$ defined as above we have
\begin{proposition}\label{difference2}
\begin{equation*}
\mathcal{N}_{\mathcal{D}}^{u}(\mu)=\mathcal{N}_{\mathcal{D}}(\mu)+\frac{1}{2}r\mu+2\mathcal{E}(\mu)+O\left(\mu^{1/3+\varepsilon}\right).
\end{equation*}
\end{proposition}

\begin{proof} In view of the definition of $\tau_{n,k}$, moving the points $(n, k-\tau_{n,k})$ down to $(n, k-1/4)$ can possibly get some of these points in the domain $\mu\mathcal{D}$ but no points out. Hence the difference between $\mathcal{N}_{\mathcal{D}}^{u}(\mu)$ and $\mathcal{N}_{\mathcal{D}}(\mu)$ is equal to twice (due to the symmetry) the number of points $(n, k-\tau_{n,k})$ in the band $\mathcal{B}_{R\mu}$ that are moved in the domain  $\mu\mathcal{D}$. There are three types of points $(n, k-\tau_{n,k})$ in this band:
\begin{enumerate}
\item $(n,k)$'s, which correspond to the case $\tau_{n,k}=0$ and definitely get in $\mu\mathcal{D}$;
\item $(n, k-\tau_{n,k})$'s with $0<\tau_{n,k}<1/4$, which may get in $\mu\mathcal{D}$;
\item $(n, k-1/4)$'s, which correspond to the case $\tau_{n,k}=1/4$ and are not moved.
\end{enumerate}

Concerning these three types of points, one key observation is that the points of the first type are all above the line passing through $O$ and $J$ (see Figure \ref{SymmH}) while the points of the third type are all below. This is because of the facts that if $rx_{n,k}\geq n+n^{1/3+\varepsilon}$ then
$k/n>G(r)/r$ and if $rx_{n,k}\leq n-n^{1/3+\varepsilon}$ then $k/n<G(r)/r$. We only prove the former fact while the latter one's proof is similar. Indeed, by Theorem \ref{thm111} and the monotonicity of $G$ if $rx_{n,k}\geq (1+c)n$ then
\begin{equation*}
\frac{k}{n}=\frac{G(n/x_{n,k})}{n/x_{n,k}}+O\left(\frac{1}{n(n+k)}\right)\geq \frac{G\left(\frac{r}{1+c}\right)}{\frac{r}{1+c}}+O\left(\frac{1}{n(n+k)}\right),
\end{equation*}
which is greater than $G(r)/r$ since $n+k\asymp \mu$. If $n+n^{1/3+\varepsilon}\leq rx_{n,k}<(1+c)n$ similarly we have
\begin{equation*}
\frac{k}{n}=\frac{G(n/x_{n,k})}{n/x_{n,k}}+O\left(n^{-1-\frac{3}{2}\varepsilon} \right)\geq \frac{G\left(\frac{r}{1+n^{-2/3+\varepsilon}}\right)}{\frac{r}{1+n^{-2/3+\varepsilon}}}+O\left(n^{-1-\frac{3}{2}\varepsilon}\right).
\end{equation*}
By the mean value theorem and a straightforward computation of $(G(x)/x)'$, we have
\begin{equation*}
0<\frac{G\left(\frac{r}{1+n^{-2/3+\varepsilon}}\right)}{\frac{r}{1+n^{-2/3+\varepsilon}}}-\frac{G(r)}{r}\gg n^{-2/3+\varepsilon}.
\end{equation*}
Combining the last two inequalities yields the desired one.

Another key observation is that any point $(n, k-\tau_{n,k})$ of the second type in the band $\mathcal{B}_{R\mu}$ is such that $|n-r\mu|\leq C'\mu^{1/3+\varepsilon}$  for some large constant $C'$.  Indeed, by Corollary \ref{approximation}, if $n-n^{1/3+\varepsilon}< rx_{n,k}< n+n^{1/3+\varepsilon}$ then
\begin{equation*}
x_{n, k}=F(n,k-\tau_{n,k})+O\left(n^{-2/3+2.5\varepsilon}\right)=\mu+O(1).
\end{equation*}
Plugging this formula of $x_{n,k}$ into the above inequality of $x_{n,k}$ yields the desired range of $n$.

As a result, the points $(n,k-\tau_{n,k})$ in the band $\mathcal{B}_{r\mu-C'\mu^{1/3+\varepsilon}}$ are only of the first type, which definitely get in $\mu\mathcal{D}$. By \eqref{error-in-band} its number is equal to $r\mu/4+\mathcal{E}(\mu)+O(\mu^{1/3+\varepsilon})$. Some of the points $(n,k-\tau_{n,k})$ with $|n-r\mu|< C'\mu^{1/3+\varepsilon}$ may get in $\mu\mathcal{D}$. Its number is of size $O(\mu^{1/3+\varepsilon})$.  The points $(n,k-\tau_{n,k})$ in $\mathcal{B}_{R\mu}\setminus \mathcal{B}_{r\mu+C'\mu^{1/3+\varepsilon}}$ are only of the third type and not moved. To sum up,  the number of points $(n, k-\tau_{n,k})$ in the band $\mathcal{B}_{R\mu}$ that are moved in the domain  $\mu\mathcal{D}$ is $r\mu/4+\mathcal{E}(\mu)+O(\mu^{1/3+\varepsilon})$. This finishes the proof.
\end{proof}

Combining Proposition \ref{difference1} and \ref{difference2} immediately yields that
\begin{theorem} \label{reduction}
\begin{align*}
\left|\mathscr{N}_{\mathscr{D}}(\mu)-\mathcal{N}_{\mathcal{D}}^{u}(\mu)+\frac{1}{2}r\mu \right|
&\leq \mathcal{N}_{\mathcal{D}}^{u}\left(\mu^+\right)-
\mathcal{N}_{\mathcal{D}}^{u}\left(\mu^-\right)\\
&\quad +2\left(\mathcal{E}\left(\mu^-\right)-\mathcal{E}\left(\mu^+\right)\right)+O\left(\mu^{0.6} \right)
\end{align*}
with $\mu^+=\mu+C\mu^{-0.4}$ and $\mu^-=\mu-C\mu^{-0.4}$.
\end{theorem}
Thus we have transferred the study of $\mathscr{N}_{\mathscr{D}}(\mu)$ to those of
$\mathcal{N}_{\mathcal{D}}^{u}(\mu)$ and $\mathcal{E}\left(\mu\right)$, which will be done in the following section.


\section{Lattice Counting and Proof of Theorem \ref{specthm}}\label{LatticeSec}

In this section we study the two associated lattice point problems, $\mathcal{N}_{\mathcal{D}}^{u}(\mu)$ and $\mathcal{E}\left(\mu\right)$, defined in \eqref{lattice-pro2} and \eqref{error-in-band} respectively. Theorem \ref{specthm} follows directly from Theorem  \ref{reduction}, Theorem \ref{theorem:no-in-D} and Corollary \ref{no-in-band}.

Recall that
\[
\mathcal{N}_{\mathcal{D}}^{u}(\mu)=\#\left(\mu\mathcal{D}\cap \left\{(m,n-1/4) : (m,n)\in\mathbb{Z}^2\right\}\right)
\]
denotes the number of points in the shifted lattice $\mathbb{Z}^2-(0,1/4)$ which lie in $\mu\mathcal{D}$. The domain $\mathcal{D}$, defined in Section \ref{reduction-sec} (see Figure \ref{SymmH}), has an area
\begin{equation*}
|\mathcal{D}|=\frac{1}{4}\left(R^2-r^2\right)\,.
\end{equation*}

\begin{theorem}\label{theorem:no-in-D} Let $0\leq r < R$. If the boundary curve of $\mathcal{D}$ has  a tangent in $J$ with rational slope (i.e. $\pi^{-1}\arccos(r/R)\in \mathbb{Q}$), then
\[
\mathcal{N}_{\mathcal{D}}^{u}(\mu)=|\mathcal{D}|\mu^2-\frac{R}{2}\mu+O\left(\mu^{\theta}(\log \mu)^\Theta\right),
\]
where
\begin{equation}\label{definition-theta}
\theta=\frac{131}{208}\approx 0.6298\,,\qquad\Theta=\frac{18627}{8320}\approx 2.2388\,.
\end{equation}
In case of an irrational slope the asymptotics remains true with the much weaker error term $O(\mu^{2/3})$.
\end{theorem}

\begin{remark}
If the tangent in J has rational slope (this includes the case $r=0$) the error term is of the same quality as the best published result in the circle problem due to Huxley \cite{Huxley:2003}.
The linear term can be explained as follows. To every lattice point one can associate an axes parallel square of volume 1 with center in the lattice point. Every such square contributes to $|\mathcal{D}|\mu^2$ the volume of its intersection with $\mu\mathcal{D}$. The points $(n,-1/4)$ with $|n|\leq R\mu$ are not counted in $\mathcal{N}_{\mathcal{D}}^{u}(\mu)$, but contribute to the volume  $\frac12R\mu+O(1)$.
\end{remark}

Since the boundary of $\mathcal{D}$ contains points with infinite curvature (the points $J$, $P_2$ and $J'$,  $P_2'$) standard results are not directly applicable. But see \cite{nowak:1980} for a lattice point counting problem in a non-convex domains with cusps and unbounded curvature. Our proof uses the following deep result of M.N. Huxley.
\begin{proposition}\label{proposition:Huxley} Let $M,N,C_1,C_2,C_3,C_4\geq 2$ be real parameters and $F:[1,2]\to\mathbb{R}$ a three times continuously differentiable function satisfying
\[C_j^{-1}\leq|F^{(j)}(x)|\leq C_j\]
for $j=1,2,3$. Denote by $\rho(x)=[x]-x+1/2$ the row-of-teeth function and by $\theta$, $\Theta$ the constants defined in (\ref{definition-theta}). Then there is a constant $B$ which depends only on $C_1,C_2,C_3$ and $C_4$, such that
\begin{equation*}
\Big|\sum_{M\leq m\leq M_2\leq 2M}\rho\left(NF\left(\frac mM\right)\right)\Big|\leq B (MN)^{\theta/2}\log(MN)^{\Theta}
\end{equation*}
provided that
\begin{equation}\label{condition:Huxley}
C_4^{-1}(MN)^{\frac{141}{10}}\log(MN)^{\frac{1083}{280}}\leq M^{\frac{164}{5}}\leq C_4(MN)^{\frac{181}{10}}\log(MN)^{\frac{2907}{1400}}\,.
\end{equation}

\end{proposition}
\begin{proof} This is Case A of Proposition 3 in \cite{Huxley:2003}.
\end{proof}

\begin{remark} In contrast to van der Corput's classical estimate (\ref{corput-second-derivative-estimate}) the proposition uses a condition on the first derivative. In our application $|F'(x)|$ becomes large if we count lattice points near to the boundary point $\mu J$ along lines parallel to the axes. To avoid this we count them on lines parallel to the tangent. This is only possible if the tangent in $J$ has rational slope.
\end{remark}

\begin{proof}[Proof of Theorem \ref{theorem:no-in-D}]
Slightly more general we count points in the shifted lattice $\mathbb{Z}^2-(0,c)$ with $c\in[0,1/2)$.
The number $\mathcal{N}_{\mathcal{D}}^{u}(\mu)$ is twice the number of shifted lattice points in the positive quadrant, if points on the $y$-axis are counted with weight 1/2. Divide  $\mathcal{D}\cap[0,\infty)^2$ in domains
\begin{align*}
\mathcal{D}_1&:=\{(x,y)\in\mathcal{D}: 0\leq x\leq R, 0< y\leq G(r)\},\\
\mathcal{D}_2&:=\{(x,y)\in\mathcal{D}: 0\leq x\leq r,y>G(r)\}.
\end{align*}
See Figure \ref{decompositionD} for these domains.
\begin{figure}[ht]
\centering
\includegraphics[width=0.6\textwidth]{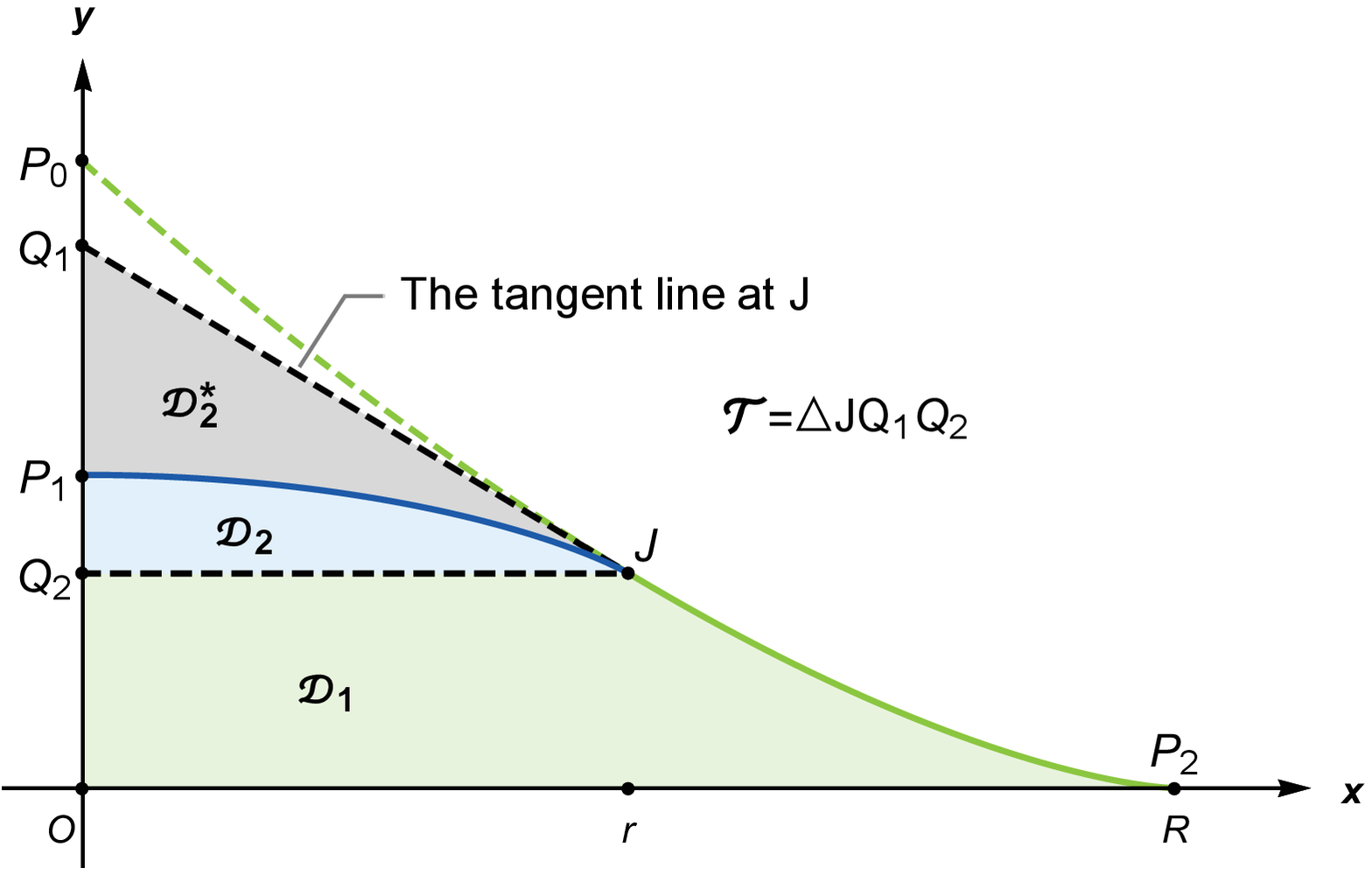} 
\caption{A decomposition of $\mathcal{D}$ in the first quadrant.}
\label{decompositionD}
\end{figure}

The rational case of Theorem \ref{theorem:no-in-D} follows if we prove that
\begin{align}
\mathcal{N}_{\mathcal{D}_1}^{u}(\mu) &=|\mathcal{D}_1|\mu^2-(1/2-c)R\mu+L_{12}+O\left(\mu^{\theta}(\log \mu)^\Theta\right),\label{ND1}\\
\mathcal{N}_{\mathcal{D}_2}^{u}(\mu) &=|\mathcal{D}_2|\mu^2-L_{12}+O\left(\mu^{\theta}(\log \mu)^\Theta\right),\label{ND2}
\end{align}
where $L_{12}=\mu r\,\rho(\mu G(r)+c)$ describes the contribution of the line segment separating $\mathcal{D}_1$ from $\mathcal{D}_2$. While (\ref{ND1}) is true in general, we prove (\ref{ND2}) in the irrational case only with the weaker error term $O(\mu^{2/3})$.

In $\mu\mathcal{D}_1$ we count lattice points along lines parallel to the $x$-axis. Denote by $H:[0,G(r)]\to[r,R]$ the inverse function of $G$ restricted to $[r,R]$. Since points on the $y$-axis are counted with weight 1/2, and $[x]+1/2=x+\rho(x)$, one finds
\begin{align*}
\mathcal{N}_{\mathcal{D}_1}^{u}(\mu)&=\sum_{0<n-c\leq \mu G(r)}\!\Big(\Big[\mu H\left(\frac{n-c}{\mu}\right)\Big]+\frac12\Big)\\
&=\sum_{\frac12<n\leq \mu G(r)+c}\!\mu H\left(\frac{n-c}{\mu}\right)+\!\!\sum_{0<n\leq \mu G(r)+c}\rho\left(\mu H\left(\frac{n-c}{\mu}\right)\right).
\end{align*}
Euler's summation formula
\begin{align*}
\sum_{a<n\leq b}f(n)=\int_a^b f(x)\,\textrm{d}x+\rho(b)f(b)-\rho(a)f(a)-\int_a^bf'(x)\rho(x)\,\textrm{d}x
\end{align*}
is used to calculate the first sum. The first integral gives the main term
\begin{align*}
\int_{1/2}^{\mu G(r)+c}\mu H\big(\frac{x-c}{\mu}\big)\,\textrm{d}x&=\mu^2\int_{(1/2-c)/\mu}^{G(r)}H(x)\,\textrm{d}x\\
&=|\mathcal{D}_1|\mu^2-(1/2-c)R\mu+O(1)\,.
\end{align*}
By the second mean value theorem and Lemma \ref{Lemma-function-H} the second integral is bounded by
\begin{align*}
\int_{1/2}^{\mu G(r)+c}H'\big(\frac{x-c}{\mu}\big)\,\rho(x)\,\textrm{d}x\ll\sup_{(1/2-c)/\mu\leq y\leq G(r)}\left|H'(y)\right|
\ll \mu^{1/3}\,.
\end{align*}
Together we obtain
\begin{align*}
\mathcal{N}_{\mathcal{D}_1}^{u}(\mu) &=|\mathcal{D}_1|\mu^2-(1/2-c)R\mu+L_{12}\\
&\qquad +\sum_{0<n\leq \mu G(r)+c}\rho\left(\mu H\left(\frac{n-c}{\mu}\right)\right)+O(\mu^{1/3})\,.
\end{align*}
For $n\leq V:=\mu^{\theta}$ the $\rho$-sum is estimated trivially. This contributes $O(\mu^\theta)$ to $\mathcal{N}_{\mathcal{D}_1}^{u}(\mu)$. The remaining sum is divided in sums of the form
\begin{align}\label{psisum}
\sum_{M\leq m\leq M'\leq 2M}\rho\left(\mu H\left(\frac{m-c}{\mu}\right)\right)\,,
\end{align}
where $M=2^jV\ll\mu$. To apply Proposition \ref{proposition:Huxley} set
\[\mu H\left(\frac{m-c}{\mu}\right)=NF\left(\frac{m}{M}\right)\quad\mbox{with}\quad F(x)=\left(\frac{\mu}{M}\right)^{2/3}H\left(\frac M\mu x-\frac{c}{\mu}\right)\]
and $N=M^{2/3}\mu^{1/3}$. By Lemma \ref{Lemma-function-H} $|F^{(j)}(x)|\asymp 1$ for $x\in[1,2]$ and $j=1,2,3$. The condition (\ref{condition:Huxley}) is satisfied since $V\leq M\ll \mu$. This yields the bound $(M^{5/3}\mu^{1/3})^{\theta/2}(\log\mu)^{\Theta}$ for (\ref{psisum}). Summing over $M=2^jV\ll\mu$ gives (\ref{ND1}). Note that this already completes the treatment of the special case $r=0$.

If $r>0$ we have to deal with $\mathcal{N}_{\mathcal{D}_2}^{u}(\mu)$. First we assume that the tangent in $J$ has rational slope. Hence $G'(r)=-a/q<0$ with $a$, $q$ relatively prime. The number of shifted lattice points in $\mu\mathcal{D}_2$ is equal to the number of shifted lattice points in the triangle $\mu\mathcal{T}$ minus the number of shifted lattice points in $\mu\mathcal{D}^*_2$, where
\begin{align*}
\mathcal{T}&:=\big\{(x,y)\in\mathbb{R}^2: 0\leq x<r, G(r)< y\leq G(r)+\frac aq(r-x)\big\},\\
\mathcal{D}^*_2&:=\big\{(x,y)\in\mathbb{R}^2: 0\leq x<r,G(x)< y\leq G(r)+\frac aq(r-x)\big\}.
\end{align*}
In case of $\mu\mathcal{T}$ and $\mu\mathcal{D}_2^*$ it is easier to count points on the $y$-axis with full weight. Then
\begin{equation}\label{label-ND2}
\mathcal{N}_{\mathcal{D}_2}^{u}(\mu)=\mathcal{N}_{\mathcal{T}}^{u}(\mu)-\mathcal{N}_{\mathcal{D}_2^*}^{u}(\mu)-\mu (G(0)-G(r))/2+O(1)\,.
\end{equation}

In $\mu\mathcal{D}_2^*$ we count points along the lines $g_t: ax+q(y+c)=t$, $t\in \mathbb{Z}$. Note that $g_t$ contains points of the shifted lattice $\mathbb{Z}^2-(0,c)$ if and only if $t$ is an integer.  The line $g_t$ intersects the lower boundary curve of $\mu\mathcal{D}_2^*$ between $(0,\mu G(0))$ and $(\mu r,\mu G(r))$ in a unique point if $t\in[\mu q\beta,\mu q\gamma]$, where
\[\beta=G(0)+\frac{c}{\mu}\,,\qquad\gamma=G(r)+\frac aq r+\frac{c}{\mu}\,.\]
Define a function $T$ by writing the $x$-coordinate of the intersection point as $\mu T(t/(\mu q))$. The defining equation of $T$ reads
\begin{align}\label{definition-T}
G(T(y))+\frac aq T(y)+\frac{c}{\mu}=y \qquad(y\in[\beta,\gamma])\,.
\end{align}
The strictly increasing function $T$ maps $[\beta,\gamma]$ to $[0,r]$.
For every $t_0\in\{0,\dots,q-1\}$ choose $x_0\in\{0,\dots,q-1\}$ such that $ax_0\equiv t_0\pmod q$. If $t\equiv t_0\pmod q$ the lattice points on $g_t$ are the points $(m,n-c)$ with $m=x_0+qk$, $k\in\mathbb{Z}$.
Hence the number of shifted lattice points in $\mu\mathcal{D}_2^*\cap g_t$ is equal to the number of integers $k$ such that $-x_0/q\leq k<(\mu T(t/(\mu q))-x_0)/q$.
Since the number of integers in $[a,b)$ is $[-a]-[-b]=b-a-\rho(-b)+\rho(-a)$ this number is
\[
\frac{\mu}{q}T\left(\frac{t}{\mu q}\right)-\rho\left(-\frac\mu q T\Big(\frac{t}{\mu q}\Big)+\frac{x_0}q\right)+\rho\left(\frac{x_0}q\right)\,.
\]
This yields
\begin{align}\label{ND2lable1}
\mathcal{N}_{\mathcal{D}_2^*}^{u}(\mu)=\sum_{\mu q\beta<t\leq\mu q \gamma}
\frac{\mu}{q}T\left(\frac{t}{\mu q}\right)-S_1+S_2\,
\end{align}
with
\begin{align*}
S_1&:=\sum_{t_0=0}^{q-1}\sum_{\mu  \beta-t_0/q<k\leq\mu  \gamma-t_0/q}
\rho\left(-\frac\mu q T\left(\frac{k}{\mu }+\frac{t_0}{\mu q}\right)+\frac{x_0}q\right),\\
S_2&:=\sum_{t_0=0}^{q-1}\rho\left(\frac{x_0}{q}\right)\big(\mu(\gamma-\beta)+O(1)\big).\\
\end{align*}
Using the relation
\begin{align}\label{complete-psi-sum}
\sum_{k=0}^{q-1}\rho\big((x+k)/q\big)=\rho(x),
\end{align}
$S_2$ simplifies to
\[S_2=\mu(\gamma-\beta)/2+O(q)\,.\]
Euler's summation formula applied to the first sum in (\ref{ND2lable1}) yields
\[
|\mathcal{D}_2^*|\mu^2+\rho(\mu q\gamma)\frac{\mu}{q}r-\frac1{q^2}\int_{\mu q\beta}^{\mu q \gamma}T'\Big(\frac{x}{\mu q }\Big)\rho(x)\,\textrm{d}x\,.
\]
By the second mean value theorem, Lemma \ref{Lemma-function-T} and (\ref{label-T}) the last integral is bounded by
\[
\int_{\mu q\beta}^{\mu q\gamma-1}T'\big(\frac{x}{\mu q}\big)\,\rho(x)\,\textrm{d}x+\int_{\mu q\gamma-1}^{\mu q\gamma}T'\big(\frac{x}{\mu q}\big)\,\textrm{d}x
\]
\[
\ll \sup_{\beta\leq y\leq\gamma-1/(\mu q)}T'(y)+\mu q\left(T(\gamma)-T\left(\gamma-\frac1{\mu q}\right)\right)\ll\mu^{1/3}\,.
\]
In the inner sum of $S_1$ we estimate the terms with $[\mu \gamma]-V<k\leq \mu  \gamma$ trivially.
The remaining sum is divided in sums of the form
\[
\sum_{[\mu \gamma]-M'\leq t\leq [\mu \gamma]-M}\!\rho\left(-\frac\mu q T\left(\frac{k}{\mu}+\frac{t_0}{\mu q}\right)+\frac{x_0}{q}\right)=\!\sum_{M\leq m\leq M'\leq 2M}\!\rho\left(NF\left(\frac mM\right)\right),
\]
where $M=2^jV\ll \mu$, $N=\mu^{1/3}M^{2/3}q^{-1}$ and
\[F(x)=-\Big(\frac{\mu}{M}\Big)^{2/3}T\left(\gamma-\frac{M}{\mu}x+\frac{c_0}\mu\right)+\frac{x_0}{qN}\qquad(x\in[1,2])\]
with $c_0=t_0/q+[\gamma\mu]-\gamma\mu$.
By Lemma \ref{Lemma-function-T} $|F^{(j)}(x)|\asymp 1$ for $j=1,2,3$. The condition (\ref{condition:Huxley}) of Proposition \ref{proposition:Huxley} is satisfied since $V\leq M\ll \mu$. This yields the bound $(\mu^{1/3}M^{5/3})^{\theta/2}(\log\mu)^{\Theta}$. Summing over $2^jV\ll\mu$ gives $S_2\ll\mu^{\theta}(\log\mu)^\Theta$. Together this proves
\begin{align}\label{ND2-label2}
\mathcal{N}_{\mathcal{D}_2^*}^{u}(\mu)=|\mathcal{D}_2^*|\mu^2+r\frac\mu q\rho(\mu q \gamma)+\frac{\mu}{2}(\gamma-\beta)+O(\mu^{\theta}(\log \mu)^{\Theta})\,.
\end{align}

To evaluate  $\mathcal{N}_{\mathcal{T}}^{u}(\mu)$ we start from
\[
\mathcal{N}_{\mathcal{T}}^{u}(\mu)=\sum_{0\leq n<\mu r}\left(\frac aq(\mu r-n)+\rho\left(\mu\gamma-\frac aq n\right)-\rho\left(\mu G(r)+c\right)\right).
\]
The last sum is $-L_{12}+O(1)$. Using (\ref{complete-psi-sum}) the second sum is $r\frac{\mu}{q}\rho(\mu q \gamma)+O(q)$. The
following version of Euler's summation formula
\begin{align*}
\sum_{a\leq n<b}f(n)=\int_a^b f(x)\,\textrm{d}x-\rho(-b)f(b)+\rho(-a)f(a)-\int_a^bf'(x)\rho(x)\,\textrm{d}x
\end{align*}
is used to calculate the first sum. Its value is $|\mathcal{T}|\mu^2+\frac{a}{q}\frac\mu 2+O(1)$. Hence
\[
\mathcal{N}_{\mathcal{T}}^{u}(\mu)=|\mathcal{T}|\mu^2+\frac aq\frac \mu2+r\frac\mu q\rho(\mu q\gamma)-L_{12}+O(1)\,.
\]
Together with (\ref{label-ND2}) and (\ref{ND2-label2}) this proves (\ref{ND2}) and completes the proof of Theorem \ref{theorem:no-in-D} in the rational case.

In the irrational case we prove (\ref{ND2}) with the weaker error term $O(\mu^{2/3})$. Since points on the $y$-axis are counted with weight 1/2 an application of Euler's summation formula yields
\begin{align*}
\mathcal{N}^{u}_{\mathcal{D}_2}(\mu)
&= \sum_{0<m\leq \mu r}\left(\Big[\mu G\left(\frac{m}{\mu}\right)+c\Big]-\Big[\mu G\left(r\right)+c\Big]\right)\\
&\qquad\qquad\qquad+\frac\mu 2\big(G(0)-G(r)\big)+O(1)\\
&=|\mathcal{D}_2|\mu^2-L_{12}+\sum_{1\leq m\leq \mu r}\rho\left(\mu G\left(\frac{m}{\mu}\right)+c\right)+O(1)\,.
\end{align*}
Van der Corput's second derivative estimate \cite{corput:1923}
\begin{align}\label{corput-second-derivative-estimate}
\sum_{M_1\leq m\leq M_2 }\rho(f(m))\ll\int_{M_1}^{M_2}|f''(x)|^{1/3}\,\textrm{d}x+\max_{x\in[M_1,M_2]}|f''(x)|^{-1/2}
\end{align}
gives the bound $O(\mu^{2/3})$ for the $\rho$-sum.
\end{proof}


\begin{corollary}\label{no-in-band} Let $0<r<R$.
If the boundary curve of $\mathcal{D}$ has  a tangent in $J$ with rational slope then  $\mathcal{E}(\mu)$ defined in \eqref{error-in-band} satisfies
\begin{equation}\label{assertion-no-in-band}
\mathcal{E}(\mu)=O\left(\mu^{\theta}(\log \mu)^\Theta\right)
\end{equation}
with $\theta$ and $\Theta$ as in \eqref{definition-theta}. In case of an irrational slope the weaker bound $\mathcal{E}(\mu)=O\left(\mu^{2/3} \right)$ is true.
\end{corollary}

\begin{proof}
In $\mathcal{E}(\mu)=\# \left(\mathcal{B}_{r\mu}\cap\mathbb{Z}^2 \right)-\frac{1}{4}r\mu$ we count unshifted lattice points. The number of unshifted lattice points in $\mu\mathcal D_2$ is given by (\ref{ND2}) with $c=0$. Thus in the rational case (\ref{assertion-no-in-band}) is equivalent to
\begin{align}\label{ND2+}
\mathcal{N}_{\mathcal{D}_2^+}^{u}(\mu)=|\mathcal{D}_2^+|\mu^2-L_{12}+O(\mu^\theta(\log \mu)^\Theta)\,,
\end{align}
where $\mathcal{N}_{\mathcal{D}_2^+}^{u}(\mu)$ denotes the number of unshifted lattice points in $\mu\mathcal{D}_2^+$ with
\[\mathcal{D}_2^+:=\left\{(x,y)\in\mathbb{R}^2: 0\leq x\leq r,\ G(r)<y\leq  G(x)+1/(4\mu)\right\}.\]
Repeating the proof of (\ref{ND2}) with this slightly modified domain one obtains (\ref{ND2+}). The  case of irrational slope is even easier.
\end{proof}

\begin{lemma}\label{Lemma-function-H} Let $0\leq r<R$. The inverse function $H:[0,G(r)]\to [r,R]$ of $G$ restricted to $[r,R]$ satisfies for j=1,2,3
\[H^{(j)}(y)\asymp y^{\frac23-j}\qquad(y\in (0,G(r)])\,.\]
\end{lemma}
\begin{proof} For $0\leq r\leq x\leq R$ the function $G(x)=Rg(x/R)$ satisfies
\begin{align*}
G'(x)&=-\pi^{-1}\arccos(x/R)\asymp(R-x)^{1/2}\,,\\
G''(x)&=(\pi R)^{-1}\left(1-\left(x/R\right)^2\right)^{-1/2}\asymp (R-x)^{-1/2}\,,\\
G'''(x)&=(\pi R^3)^{-1}x\left(1-(x/R)^2\right)^{-3/2}\asymp x(R-x)^{-3/2}\,
\end{align*}
and, with the positive and bounded function $h(x)=x(1-x^2)^{-1/2}\arccos(x)$,
\[3G''(x)^2-G'(x)G'''(x)=(\pi R)^{-2}\textstyle{\left(1-\left(\frac{x}{R}\right)^2\right)^{-1}}\left(3+h\textstyle{\left(\frac xR\right)}\right)\asymp(R-x)^{-1}\,.\]
Furthermore  $f(x):=G(x)(R-x)^{-3/2}$ is positive with
\[\lim_{x\to R}f(x)=R^{-5/2}2^{3/2}(3\pi)^{-1}>0\,.\]
This proves \[G(x)\asymp(R-x)^{3/2}\,.\]
Set $y=G(x)$. Then $R-x\asymp y^{2/3}$. For the inverse function $H$ one obtains
\begin{align*}
H'(y)&=\left(G'(x)\right)^{-1}\asymp(R-x)^{-1/2}\asymp y^{-1/3}\,,\\
H''(y)&=-\left(G'(x)\right)^{-3}G''(x)\asymp(R-x)^{-2}\asymp y^{-4/3}\,,\\
H'''(y)&=\left(G'(x)\right)^{-5}\left(3G''(x)^2-G'''(x)G'(x)\right)\asymp(R-x)^{-7/2}\asymp y^{-7/3}\,.
\end{align*}
\end{proof}

\begin{lemma}\label{Lemma-function-T} Let $0<r<R$. The function $T:[\beta,\gamma]\to[0,r]$ defined in (\ref{definition-T}) satisfies for j=1,2,3
\[T^{(j)}(y)\asymp (\gamma-y)^{\frac23-j}\qquad(y\in[\beta,\gamma)).\]
\end{lemma}
\begin{proof} On $[0,r]$ the function $G$ is defined by $G(x)=Rg(x/R)-rg(x/r)$. Thus
\begin{align*}
G''(x)&=\textstyle{\frac1{\pi R}\left(1-\left(\frac xR\right)^2\right)^{-1/2}-\frac1{\pi r}\left(1-\left(\frac xr\right)^2\right)^{-1/2}}\asymp (r-x)^{-1/2}\,,\\
G'''(x)&=\textstyle{\frac x\pi\left(\frac1{R^3}\left(1-\left(\frac xR\right)^2\right)^{-3/2}-\frac1{r^3}\left(1-\left(\frac xr\right)^2\right)^{-3/2}\right)}\asymp x(r-x)^{-3/2}\,
\end{align*}
and
\[G'(x)-G'(r)=-\int_x^rG''(u)\,\textrm{d}u\asymp (r-x)^{1/2}\,.\]
Since $G''(x)<0$ the function
\[f(x):=\left(G(x)-G(r)-G'(r)(x-r)\right)(r-x)^{-3/2}\]
is strictly negative with $\lim_{x\uparrow r}f(x)=-2^{3/2}/(3\sqrt{r})>-\infty$. This proves
\begin{align}\label{label-asymptotic-G}
G(x)-G(r)-G'(r)(x-r)\asymp (r-x)^{3/2}\,.
\end{align}
Set $x=T(y)$. Subtracting (\ref{definition-T}) from the equation defining $\gamma$ one obtains
\[\gamma-y=G(r)-G(x)-G'(r)(r-x)\,.\]
With (\ref{label-asymptotic-G}) this yields $\gamma-y\asymp(r-x)^{3/2}$ and
\begin{align}\label{label-T}
r-T(y)=r-x\asymp(\gamma-y)^{2/3}\,.
\end{align} Differentiating  (\ref{definition-T}) one obtains
\begin{align*}
T'(y)&=\left(G'(x)-G'(r)\right)^{-1}\asymp (r-x)^{-1/2}\asymp(\gamma-y)^{-1/3},\\
T''(y)&=-\left(G'(x)-G'(r)\right)^{-3}G''(x)\asymp(r-x)^{-2}\asymp(\gamma-y)^{-4/3},\\
T'''(y)&=\left(G'(x)-G'(r)\right)^{-5}(G''(x))^2\big(3+F(x)\big)
\asymp(\gamma-y)^{-7/3}.
\end{align*}
Here $F(x)=-G'''(x)(G'(x)-G'(r))(G''(x))^{-2}$ is a positive bounded function.
\end{proof}



\appendix

\section{Some asymptotics}


For any $c>0$ and $n\in \mathbb{N}\cup\{0\}$, if $z\geq \max\{(1+c)n, 1\}$ the Bessel functions have the asymptotics
\begin{equation}
J_n(z)=\left(2/\pi\right)^{1/2}\left(z^2-n^2\right)^{-1/4} \left(\cos\left( \pi
zg\left(\frac{n}{z}\right)-\frac{\pi}{4}\right)+O_c\left(z^{-1}\right)\right)    \label{jnasy}
\end{equation}
and
\begin{equation}
Y_n(z)=\left(2/\pi\right)^{1/2}\left(z^2-n^2\right)^{-1/4} \left(\sin\left( \pi
zg\left(\frac{n}{z}\right)-\frac{\pi}{4}\right)+O_c\left(z^{-1}\right)\right),   \label{ynasy}
\end{equation}
where $g$ is defined by \eqref{def-g}.

We sketch the proof. Recall the integral representations \cite[p. 360]{abram:1972}
\[J_n(z)=\operatorname{Re}(I_n(z))\,,\qquad Y_n(z)=\operatorname{Im}(I_n(z))-L_n(z)\,,\]
where
\begin{align*}
L_n(z):=\frac1\pi\int_0^\infty\left(e^{nt}+e^{-nt}\cos(n\pi)\right)e^{-z\sinh t}  \,\textrm{d}t
\end{align*}
and
\begin{equation*}
I_n(z):=\frac{1}{\pi}\int_0^\pi e^{i z \phi(\theta)} \,\textrm{d}\theta
\end{equation*}
with
\begin{equation*}
\phi(\theta)=\phi(z,n,\theta)=\sin\theta-\frac{n}{z}\theta.
\end{equation*}

 We first study the integral $I_n(z)$. The phase function $\phi$ has only one critical point  $\beta:=\arccos(n/z)$ in $[0, \pi]$. Applying the method of stationary phase in a sufficiently small neighborhood of $\beta$ yields the contribution
\begin{equation}
\left(2/\pi\right)^{1/2}|\phi''(\beta)|^{-1/2} e^{i(z\phi(\beta)-\pi/4)}z^{-1/2}+O_c(z^{-3/2}).\label{appendix222}
\end{equation}
To study the contribution of the domain away from $\beta$ we use integration by parts twice. The real part contributes at most $O_c(z^{-2})$ while the imaginary part is equal to
\begin{equation}
\frac{1}{\pi}\left(\frac{1}{z-n}+\frac{\cos(n\pi)}{z+n}\right)+O_c(z^{-2}).\label{appendix333}
\end{equation}

We then immediately get \eqref{jnasy} by taking the real part of $I_n(z)$.

It remains to study $L_n(z)$. If $n\leq 2/c$, by using a substitution $y=\sinh t$ and integration by parts twice, we get
\begin{equation*}
L_n(z)=\frac{1+\cos(n\pi)}{\pi z}+O_c(z^{-2}).
\end{equation*}

If $n> 2/c$ then
\begin{equation}
\int_0^\infty e^{nt}e^{-z\sinh t}  \,\textrm{d}t=\frac{1}{z-n}+O_c(z^{-3}).\label{appendix111}
\end{equation}
Indeed, by changing variables $s=(z-n)t$ we have
\begin{equation*}
\int_0^\infty e^{nt}e^{-z\sinh t}  \,\textrm{d}t=\frac{1}{z-n}\int_0^\infty e^{-s}e^{-z\sigma\left(\frac{s}{z-n} \right)}  \,\textrm{d}s,
\end{equation*}
where
\begin{equation*}
\sigma(t)=\sum_{k=1}^{\infty} \frac{t^{2k+1}}{(2k+1)!}.
\end{equation*}
Therefore, by the mean value theorem,
\begin{equation*}
\left|\int_0^\infty e^{nt}e^{-z\sinh t}  \,\textrm{d}t-\frac{1}{z-n}\right|\leq \frac{z}{z-n} \int_0^\infty e^{-s}\sigma\left(\frac{s}{z-n} \right)  \,\textrm{d}s.
\end{equation*}
After using the Gamma function to simplify the right hand side, we get a convergent geometric series which is $\ll z^{-3}$. This proves \eqref{appendix111}.

Repeating the above argument for \eqref{appendix111} (even for small $n$) yields
\begin{equation*}
\int_0^\infty e^{-nt}\cos(n\pi)e^{-z\sinh t}  \,\textrm{d}t=\frac{\cos(n\pi)}{z+n}+O(z^{-3}).
\end{equation*}

Finally, combining \eqref{appendix222}, \eqref{appendix333}, and the above asymptotics of $L_n(z)$ leads to  \eqref{ynasy} readily. This finishes the proof.



Furthermore, we use Olver's uniform asymptotic expansions of Bessel functions of large order (see \cite[p. 368]{abram:1972} or \cite{olver:1954}):
\begin{equation}
J_n(nz)\sim \left(\frac{4\zeta}{1-z^2}\right)^{1/4}\left(\frac{\mathrm{Ai}(n^{2/3}\zeta)}{n^{1/3}}
\sum_{k=0}^{\infty}\frac{a_k(\zeta)}{n^{2k}}+\frac{\mathrm{Ai}^{\prime}(n^{2/3}\zeta)}{n^{5/3}}
\sum_{k=0}^{\infty}\frac{b_k(\zeta)}{n^{2k}}\right)\label{jnuse111}
\end{equation}
and
\begin{equation}
Y_n(nz) \sim -\left(\frac{4\zeta}{1-z^2}\right)^{1/4}\left(\frac{\mathrm{Bi}(n^{2/3}\zeta)}
{n^{1/3}}\sum_{k=0}^{\infty}\frac{a_k(\zeta)}{n^{2k}}+\frac{\mathrm{Bi}^{\prime}(n^{2/3}\zeta)}{n^{5/3}}
\sum_{k=0}^{\infty}\frac{b_k(\zeta)}{n^{2k}}\right),\label{ynuse111}
\end{equation}
where $\zeta=\zeta(z)$ is given by
\begin{equation}
\frac{2}{3}(-\zeta)^{3/2}=\int_{1}^z\frac{\sqrt{t^2-1}}{t}\,\mathrm{d}t=
\sqrt{z^2-1}-\arccos\left(\frac{1}{z}\right)\label{def-zeta1}
\end{equation}or
\begin{equation}
\frac{2}{3}\zeta^{3/2}=\int_{z}^1\frac{\sqrt{1-t^2}}{t}\,\mathrm{d}t=\ln \frac{1+\sqrt{1-z^2}}{z}-\sqrt{1-z^2}.\label{def-zeta2}
\end{equation}
Here the branches are chosen so that $\zeta$ is real when $z$ is positive. $\mathrm{Ai}$ and $\mathrm{Bi}$ denotes the Airy functions of first and second kind. For the definitions and sizes of the coefficients $a_k(\zeta)$'s and $b_k(\zeta)$'s see \cite[p. 368--369]{abram:1972}, especially $a_0(\zeta)=1$.
It is easy to check the following expansions of \eqref{def-zeta1} and \eqref{def-zeta2}.
If $z\rightarrow 1+$ then
\begin{equation}\label{bound-zeta+}
(-\zeta(z))^{3/2}=\sqrt{2}(z-1)^{3/2}+\frac{11\sqrt{2}}{4}(z-1)^{5/2}+O\left((z-1)^{7/2}\right).
\end{equation}
If $z\rightarrow 1-$ then
\begin{equation}\label{bound-zeta-}
(\zeta(z))^{3/2}=\sqrt{2}(1-z)^{3/2}+\frac{9\sqrt{2}}{20}(1-z)^{5/2}+O\left((1-z)^{7/2}\right).
\end{equation}

As a consequence of Olver's asymptotics we obtain the following analogue of the 9.3.4 in \cite[p. 366]{abram:1972}.
\begin{lemma}\label{9.3.4analogue}
For any $\epsilon>0$ and all sufficiently large $n$,
\begin{enumerate}
\item if $0\leq w\leq n^{\epsilon}$ then
\begin{align*}
J_n\left(n+wn^{1/3}\right)&=2^{1/3}n^{-1/3}\mathrm{Ai}\left(-2^{1/3}w\right)+O\left(n^{-1+2.25\epsilon}\right),\\
Y_n\left(n+wn^{1/3}\right)&=-2^{1/3}n^{-1/3}\mathrm{Bi}\left(-2^{1/3}w\right)+O\left(n^{-1+2.25\epsilon}\right);
\end{align*}

\item if $-n^{\epsilon}\leq w\leq 0$ then
\begin{align*}
J_n\left(n+wn^{1/3}\right)&=2^{1/3}n^{-1/3}\mathrm{Ai}\left(-2^{1/3}w\right)\left(1+O\left(n^{-2/3+2.5\epsilon}\right)\right),\\
Y_n\left(n+wn^{1/3}\right)&=-2^{1/3}n^{-1/3}\mathrm{Bi}\left(-2^{1/3}w\right)\left(1+O\left(n^{-2/3+2.5\epsilon}\right)\right).
\end{align*}
\end{enumerate}
\end{lemma}

\begin{proof}
We may assume that $|w|\geq C$ for a large constant $C$, otherwise all desired formulas follow from  9.3.4 in \cite[p. 366]{abram:1972} since $\mathrm{Ai}(r)\asymp_C 1$ and $\mathrm{Bi}(r)\asymp_C 1$ if $0\leq r\leq C$.

Let us consider the case $-n^{\epsilon}\leq w\leq -C$. Set $z=1+wn^{-2/3}$. By \eqref{jnuse111} we have
\begin{equation*}
\begin{split}
J_n&\left(n+wn^{1/3}\right)=J_n(nz)= \\
                &\left(\frac{4\zeta}{1-z^2}\right)^{1/4}\frac{1}{n^{1/3}}\left(\mathrm{Ai}\left(n^{2/3}\zeta\right)
\left(1+O\left(n^{-2}\right) \right)+n^{-4/3}\mathrm{Ai}^{\prime}\left(n^{2/3}\zeta\right)O(1)\right),
\end{split}
\end{equation*}
where $\zeta=\zeta(z)$, determined by \eqref{def-zeta2}, is positive and satisfies by \eqref{bound-zeta-}
\begin{equation*}
n^{2/3}\zeta=-2^{1/3}w+\frac{3}{10}2^{1/3}w^2 n^{-2/3}\left(1+O\left(|w|n^{-2/3}\right)\right).
\end{equation*}
Thus
\begin{equation}
\left(\frac{4\zeta}{1-z^2}\right)^{1/4}=2^{1/3}\left(1+O\left(|w|n^{-2/3}\right) \right).\label{firstfactor}
\end{equation}
Since it is known (\cite[p. 448]{abram:1972}) that for large $r$
\begin{equation*}
\mathrm{Ai}(r)=\frac{1}{2}\pi^{-1/2}r^{-1/4}e^{-\frac{2}{3}r^{3/2}}\left(1+O\left(r^{-3/2}\right)\right)
\end{equation*}
and
\begin{equation*}
\mathrm{Ai}'(r)=-\frac{1}{2}\pi^{-1/2}r^{1/4}e^{-\frac{2}{3}r^{3/2}}\left(1+O\left(r^{-3/2}\right)\right)
\end{equation*}
we  have
\begin{equation*}
\left|\mathrm{Ai}'\left(n^{2/3}\zeta\right)\right|\ll \left(-2^{1/3}w\right)^{1/4}e^{-\frac{2}{3}\left(-2^{1/3}w\right)^{3/2}}\asymp \left(-2^{1/3}w\right)^{1/2}\mathrm{Ai}\left(-2^{1/3}w\right)
\end{equation*}
and, by the mean value theorem,
\begin{equation*}
\mathrm{Ai}\left(n^{2/3}\zeta\right)=\mathrm{Ai}\left(-2^{1/3}w\right)+
O\left(\mathrm{Ai}\left(-2^{1/3}w\right)|w|^{2.5}n^{-2/3} \right).
\end{equation*}
Collecting the above three estimates and plugging them into the above formula of $J_n\left(n+wn^{1/3}\right)$ gives the desired formula of $J_n$ in the case  $-n^{\epsilon}\leq w\leq -C$. Almost the same argument gives the formula of $Y_n$.

In the case $C\leq w\leq n^{\epsilon}$  we again use the asymptotics \eqref{jnuse111} and
\eqref{ynuse111} with $z=1+wn^{-2/3}$. The corresponding $\zeta=\zeta(z)$, determined by \eqref{def-zeta1}, is negative and satisfies by \eqref{bound-zeta+}
\begin{equation*}
-n^{2/3}\zeta=2^{1/3}w+\frac{11}{6}2^{1/3}w^2 n^{-2/3}\left(1+O\left(wn^{-2/3}\right)\right),
\end{equation*}
which leads to \eqref{firstfactor}. Since $\mathrm{Ai}(-r)=O(r^{-1/4})$ and $\mathrm{Ai}'(-r)=O(r^{1/4})$ (see \cite[p. 448--449]{abram:1972}), we get
\begin{equation*}
\mathrm{Ai}\left(n^{2/3}\zeta\right)\ll w^{-1/4}, \quad \quad \mathrm{Ai}'\left(n^{2/3}\zeta\right)\ll w^{1/4}
\end{equation*}
and
\begin{equation*}
\mathrm{Ai}\left(n^{2/3}\zeta\right)=\mathrm{Ai}\left(-2^{1/3}w\right)+O\left(w^{2.25}n^{-2/3} \right).
\end{equation*}
Collecting these estimates and plugging them into \eqref{jnuse111} gives the desired formula of $J_n$ in the case  $C\leq w\leq n^{\epsilon}$. A similar argument gives the formula of $Y_n$.
\end{proof}

Finally, we collect two well-known asymptotic formulas for the Airy functions (see for example \cite[p. 448--449]{abram:1972}). For $r>0$
\begin{equation*}
\mathrm{Ai}(-r)=\pi^{-1/2}r^{-1/4}\left(\sin\left(\frac{2}{3}r^{3/2}+\frac{\pi}{4}\right)+O\left(r^{-3/2}\right)\right)
\end{equation*}
and
\begin{equation*}
\mathrm{Bi}(-r)=\pi^{-1/2}r^{-1/4}\left(\cos\left(\frac{2}{3}r^{3/2}+\frac{\pi}{4}\right)+O\left(r^{-3/2}\right)\right).
\end{equation*}

%

\end{document}